\documentclass[11pt]{amsart}
\usepackage{graphicx}
\usepackage[mathcal]{euscript}
\usepackage{float}
\usepackage{cite}
\usepackage{verbatim}

\restylefloat{figure}

\newtheorem{theorem}{Theorem}[section]
\newtheorem{lemma}[theorem]{Lemma}
\newtheorem{corollary}[theorem]{Corollary}
\newtheorem{proposition}[theorem]{Proposition}

\theoremstyle{definition}

\theoremstyle{remark}

\numberwithin{equation}{section}

\def\girth{\operatorname{girth}}
\def\id{\operatorname{id}}

\def\im{\operatorname{im}}
\def\ker{\operatorname{ker}}
\def\red{\operatorname{red}}
\def\vup{\nu_{\uparrow}}
\def\vdown{\nu_{\downarrow}}
\def\rank{\operatorname{rank}}
\def\Tor{\operatorname{Tor}}
\def\co{\colon\thinspace}

\usepackage{tikz}

\usetikzlibrary{arrows,shapes,decorations,backgrounds,patterns}

\pgfdeclarelayer{background}
\pgfdeclarelayer{background2}
\pgfdeclarelayer{background2a}
\pgfdeclarelayer{background2b}
\pgfdeclarelayer{background3}
\pgfdeclarelayer{background4}
\pgfdeclarelayer{background5}
\pgfdeclarelayer{background6}
\pgfdeclarelayer{background7}

\pgfsetlayers{background7,background6,background5,background4,background3,background2b,background2a,background2,background,main}

\definecolor{lsupurple}{RGB}{70,29,124}
\definecolor{lsugold}{RGB}{253,208, 35}

\begin{document}

\title{Chromatic homology, Khovanov homology, and torsion }

\author{Adam M. Lowrance}
\address{Department of Mathematics and Statistics\\
Vassar College\\
Poughkeepsie, NY} 
\email{adlowrance@vassar.edu}

\author{Radmila Sazdanovi\'{c}}
\address{Department of Mathematics\\
North Carolina State University\\
Raleigh, NC}
\email{rsazdanovic@math.ncsu.edu}

\date{}

\begin{abstract}
In the first few homological gradings, there is an isomorphism between the Khovanov homology of a link and the categorification of the chromatic polynomial of a graph related to the link. In this article, we show that all torsion in the categorification of the chromatic polynomial is of order two, and hence all torsion in Khovanov homology in the gradings where the isomorphism is defined is of order two. We also prove that odd Khovanov homology is torsion-free in its first few homological gradings.
\end{abstract}

\maketitle

\section{Introduction} Khovanov homology is a categorification of the Jones polynomial  constructed by Khovanov in \cite{Kh:CategorificationJones}. The Khovanov homology $Kh(L)$ of a link $L$ is a finitely generated, bigraded abelian group. Experimental computations show that Khovanov homology frequently has nontrivial torsion. Torsion of order two has been studied extensively by Shumakovitch \cite{Shumakovitch:Torsion}, Asaeda and Przytycki \cite{AP:TorsionThickness}, Pabinak, Przytycki, and Sazdanovi\'{c} \cite{PPS:First}, and Przytycki and Sazdanovi\'{c} \cite{PS:Semiadequate}. Shumakovitch conjectures that the Khovanov homology of every link except disjoint unions and connected sums of unknots and Hopf links contains torsion of order two. This conjecture has been confirmed for many special cases, including alternating links and many semi-adequate links.

Much less is known about torsion of order not equal to two. Computations by Bar-Natan \cite{BN:FastKh} show that the Khovanov homology of the $(4,q)$ torus knot for $q=5, 7, 9,$ or $11$ contains torsion of order four.  Further computer computations by Bar-Natan and Green \cite{BNG:JavaKh} show that torus knots of higher braid index can have odd torsion. In a recent paper \cite{MPSWY:KhTorsion}, Mukherjee, Przytycki, Silvero, Wang, and Yang give many more examples of links whose Khovanov homology contains torsion of odd order. Experimental computations show that among knots with few crossings, knots with torsion of orders other than two in their Khovanov homology are less common than knots whose Khovanov homology contains only torsion of order two.

Theorem \ref{theorem:KhTorsion} gives a partial explanation to this observation, at least in the first few and last few homological gradings of Khovanov homology. The {\em all-$A$ state graph} $G_A(D)$ of $D$ is a graph obtained from the all-$A$ Kauffman state of $D$. See Section \ref{section:background} for more on Kauffman states and Section \ref{section:mainresults} for the precise construction of state graphs. The {\em girth} of $G_A(D)$ is the length of the shortest cycle in $G_A(D)$. See Figure \ref{figure:CrossingSign} for our conventions on positive and negative crossings.

\begin{theorem}
\label{theorem:KhTorsion}
Let $D$ be a link diagram with $c_-$ negative crossings such that its all-$A$ state graph $G_A(D)$ has  girth $g$ of at least two. If $-c_- \leq i \leq -c_- + g-1$, then all torsion in $Kh^{i,j}(D)$ is of order two.
\end{theorem}
Since $Kh^{i,j}(D)=0$ when $i<-c_-$, the gradings where Theorem \ref{theorem:KhTorsion} apply are the first few homological gradings where the Khovanov homology is nonzero. An analogous statement holds for the last few homological gradings of $Kh(D)$. Let $g'$ be the girth of the all-$A$ state graph of the mirror of $D$ (or equivalently of the all-$B$ state graph of $D$). Theorem \ref{theorem:KhTorsion} and the relationship between the Khovanov homology of a link and its mirror then implies that all torsion in $Kh^{i,j}(D)$ is of order two when $c_+ - g'+1 \leq i \leq c_+$, where $c_+$ is the number of positive crossings in $D$. Moreover $Kh^{i,j}(D) = 0$ if $i> c_+$.
\begin{figure}[h]
$$\begin{tikzpicture}[scale=.8, >=triangle 45]
\draw[->] (-1,-1) -- (1,1);
\draw[->] (-.25,.25) -- (-1,1);
\draw (.25,-.25) -- (1,-1);
\draw (0,-1.5) node{\Large{$+$}};

\begin{scope}[xshift = 5cm]
\draw [->] (1,-1) -- (-1,1);
\draw [->] (.25,.25) -- (1,1);
\draw (-.25,-.25) --(-1,-1);
\draw (0,-1.5) node{\Large{$-$}};
\end{scope}
\end{tikzpicture}$$
\caption{Positive and negative crossings in a link diagram.}
\label{figure:CrossingSign}
\end{figure}
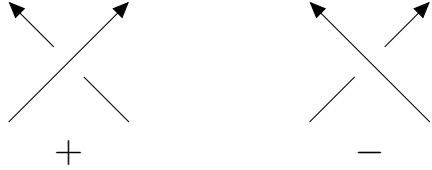

Ozsv\'ath, Rasmussen, and Szab\'o \cite{ORS:OddKh} define the odd Khovanov homology $Kh_{\text{odd}}(L)$ of a link $L$, a categorification of the Jones polynomial that agrees with Khovanov homology with $\mathbb{Z}_2$ coefficients, but differs with $\mathbb{Z}$ coefficients. In Section \ref{section:oddKh}, we prove a version of Theorem \ref{theorem:KhTorsion} for odd Khovanov homology.
\begin{theorem}
\label{theorem:OddTorsionFree}
Let $D$ be a link diagram with $c_-$ negative crossings such that its all-$A$ state graph $G_A(D)$ is planar and has girth $g$ of at least two. If $-c_-\leq i \leq -c_- + g-1$, then $Kh_{\text{odd}}^{i,j}(D)$ has no torsion.
\end{theorem}

Helme-Guizon and Rong \cite{HGR:Chromatic} define a categorification of the chromatic polynomial of a graph $G$, which we call the {\em chromatic homology} of $G$. One can view this theory as a comultiplication-free version of Khovanov homology. The Khovanov homology of a link and the chromatic homology of the all-$A$ state graph of a diagram of the link are isomorphic in certain bigradings. Theorem \ref{theorem:KhTorsion} is a consequence of this relationship together with the following result on chromatic homology.

\begin{theorem}
\label{theorem:ChromaticTorsion}
All torsion in the chromatic homology of a graph is of order two.
\end{theorem}
Theorem \ref{theorem:ChromaticTorsion} is proved in two major steps. First, we show that chromatic homology cannot have torsion of odd order by examining a version of Lee's spectral sequence \cite{Lee:Endo} for chromatic homology with $\mathbb{Z}_p$ coefficients where $p$ is an odd prime. Next, we prove that the only torsion of order $2^k$ in chromatic homology is in fact of order two. In order to achieve this second step, we define new maps $\vdown$, $\vup$, and $d_T$ on the chromatic complex with $\mathbb{Z}_2$ coefficients, each of which induces a map on chromatic homology with $\mathbb{Z}_2$ coefficients. These maps are inspired by similar maps on the Khovanov complex with $\mathbb{Z}_2$ coefficients defined by Shumakovitch \cite{Shumakovitch:Torsion} and Turner \cite{Turner:KhTwo}. We relate the induced maps to the differential in the $\mathbb{Z}_2$-Bockstein spectral sequence for chromatic homology to prove the desired result.

Chmutov, Chmutov, and Rong \cite{CCR:Knight} show that the chromatic homology of a graph $G$ with rational coefficients is determined by the chromatic polynomial of $G$. The following theorem generalizes their result to integer coefficients.

\begin{theorem}
\label{theorem:ChromaticPoly}
The chromatic homology of a graph $G$ with integer coefficients is determined by the chromatic polynomial of $G$.
\end{theorem}
The version of chromatic homology studied in this paper uses the algebra $\mathcal{A}_2=\mathbb{Z}[x]/(x^2)$. Helme-Guizon and Rong \cite{HR:Arbitrary} show that if an algebra $\mathcal{A}$ satisfies certain mild conditions, then there is a categorification of the chromatic polynomial associated to $\mathcal{A}$. Pabinak, Przytycki, and Sazdanovi\'{c} \cite{PPS:First} show that Theorem \ref{theorem:ChromaticPoly} fails for algebras other than $\mathcal{A}_2$; in particular, it fails for the algebra $\mathcal{A}_3=\mathbb{Z}[x]/(x^3)$.

This article is organized as follows. In Section \ref{section:background}, we recall the constructions of Khovanov homology and chromatic homology. In Section \ref{section:odd}, we prove that chromatic homology has no torsion of odd order. In Section \ref{section:even}, we show that the only possible torsion in chromatic homology is of order two. We also show Theorem \ref{theorem:ChromaticPoly}. In Section \ref{section:mainresults}, we define state graphs, recall the connection between Khovanov and chromatic homology, and prove Theorem \ref{theorem:KhTorsion}. In Section \ref{section:oddKh}, we recall the definition of odd Khovanov homology and prove Theorem \ref{theorem:OddTorsionFree}.\medskip

\noindent{\bf Acknowledgement:} The authors are thankful for helpful conversations with and guidance from both Alex Shumakovitch and John McCleary. Both authors are partially supported by Simons Collaboration Grants.

\section{Background}
\label{section:background}

In this section, we review the definitions of Khovanov homology and chromatic homology. We present both Bar-Natan's cube-of-resolution approach \cite{BN:Kh} and Viro's enhanced state approach \cite{Viro:Kh}. 

\subsection{Khovanov homology}

Let $R$ be a commutative ring with identity. Usually $R$ will be the integers $\mathbb{Z}$, the rationals $\mathbb{Q}$, or $\mathbb{Z}_p$, the integers modulo $p$. A bigraded $R$-module $M$ is an $R$-module with a direct sum decomposition $M=\bigoplus_{i,j\in\mathbb{Z}} M^{i,j}$ where the summand $M^{i,j}$ is said to have bigrading $(i,j)$. Equivalently, an $R$-module is bigraded if a bigrading $(i,j)$ is assigned to each element in a chosen basis. If $M=\bigoplus_{i,j\in\mathbb{Z}} M^{i,j}$ and $N=\bigoplus_{k,\ell\in\mathbb{Z}} N^{k,\ell}$ are bigraded $R$-modules, then both the direct sum $M\oplus N$ and tensor product $M \otimes N$ (understood to be taken over $R$) are bigraded with $(M\oplus N)^{m,n} = M^{m,n}\oplus N^{m,n}$ and $(M \otimes N)^{m,n} = \bigoplus _{i+k=m,j+\ell = n}M^{i,j}\otimes N^{k,\ell}$. We also define grading shift operators $[\cdot]$ and $\{\cdot\}$ by $(M[r]\{s\})^{i,j}=M^{i-r,j-s}$ where $r$ and $s$ are integers.

Let $\{0,1\}^n$ be the $n$-dimensional hypercube with vertex set $\mathcal{V}(n)$ and edge set $\mathcal{E}(n)$.  A vertex $I=(k_1,\dots, k_n)$ of the hypercube $\{0,1\}^n$ is an $n$-tuple of $0$'s and $1$'s, and there is a directed edge $\xi$ from vertex $I$ to vertex $J$ if every entry of $I$ and $J$ are the same except one entry where $I$ is $0$ and $J$ is $1$. The height $h(I)$ of a vertex is the number of $1$'s in that vertex, that is if $I=(k_1,\dots, k_n)$, then $h(I) = \sum_{i=1}^n k_i$. If $\xi$ is an edge from $I$ to $J$, then the height of $\xi$, denoted $|\xi|$, is the height of the vertex $I$. Suppose that $I$ and $J$ disagree at the $r$-th entry. Define the sign of $\xi$ by $(-1)^{\xi} = (-1)^{ \sum_{i=1}^{r} k_i}$, that is the sign of $\xi$ is $+1$ if the number of $1$'s appearing before the $r$-th entry of $I$ is even, and $-1$ otherwise.

Let $D$ be a link diagram with crossings $c_1,\dots, c_n$. Each crossing has an $A$-resolution and a $B$-resolution as in Figure \ref{figure:resolution}. The collection of curves resulting from a resolution of each crossing is called a {\em Kauffman state} $s$ of $D$. Define $|s|$ to be the number of components of the Kauffman state $s$. When resolving a crossing, replace the crossing with a small line segment, called the {\em trace of the crossing}, connecting the two strands. 

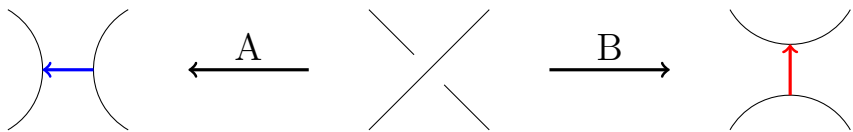
\begin{figure}[h]
$$\begin{tikzpicture}[scale=.8]
\draw (-1,-1) -- (1,1);
\draw (-1,1) -- (-.25,.25);
\draw (.25,-.25) -- (1,-1);
\draw (-3,0) node[above]{\Large{A}};
\draw[->,very thick] (-2,0) -- (-4,0);
\draw (3,0) node[above]{\Large{B}};
\draw[->,very thick] (2,0) -- (4,0);
\draw (-5,1) arc (120:240:1.1547cm);
\draw (-7,-1) arc (-60:60:1.1547cm);
\draw (5,1) arc (210:330:1.1547cm);
\draw (7,-1) arc (30:150:1.1547cm);
\draw[->,blue,very thick,] (-5.57735,0) -- (-6.4226,0);
\draw[<-,red,very thick] (6,0.422625) -- (6,-0.422625);
\end{tikzpicture}$$
\caption{The resolutions of a crossing and their traces in a link diagram. The orientations of the traces are important for the construction of odd Khovanov homology in Section \ref{section:oddKh}.} 
\label{figure:resolution}
\end{figure}

Define $\mathcal{A}_2=R[x]/(x^2) = R \oplus Rx$. In the construction of Khovanov homology, we assign a bigrading to $\mathcal{A}_2$ where $1$ has bigrading $(0,1)$  and $x$ has bigrading $(0,-1)$. The first entry in the bigrading is called the {\em homological grading}, and the second entry is the {\em polynomial grading}. Let $I=(k_1,\dots, k_n)$ be a vertex of the hypercube $\{0,1\}^n$, and define $D(I)$ to be the Kauffman state of $D$ with an $A$-resolution at crossing $c_i$ if $m_i=0$ and with a $B$-resolution at crossing $c_i$ if $m_i=1$. Define $C(D(I))=\mathcal{A}_2^{\otimes|D(I)|}[h(I)]\{h(I)\}$, where each tensor factor of $\mathcal{A}_2$ is understood to be associated to a component of $D(I)$.

Define $R$-linear maps by
\begin{equation}
\label{eq:edgemaps}
\begin{array}{l c l}
m\co\mathcal{A}_2\otimes \mathcal{A}_2\to \mathcal{A}_2 & \quad &
     m\co\begin{cases}
      1\otimes 1 \mapsto 1 &
      1\otimes x\mapsto x \\
      x \otimes 1\mapsto x &
      x \otimes x \mapsto 0
    \end{cases}\\
    \Delta\co \mathcal{A}_2\to \mathcal{A}_2\otimes \mathcal{A}_2 & \quad &
   \Delta\co\begin{cases}
      1 \mapsto 1\otimes x + x \otimes 1 &\\
      x \mapsto x\otimes x. &
    \end{cases}
  \end{array}
\end{equation}
Suppose $\xi$ is an edge in $\mathcal{E}(n)$ from $I$ to $J$. If $|D(J)| = |D(I)| -1$, then define $d_\xi\co C(D(I))\to C(D(J))$ to be multiplication $m$ on the two tensor factors of $\mathcal{A}_2$ corresponding to the components being merged and the identity on all other tensor factors of $\mathcal{A}_2$. If $|D(J)| = |D(I)| +1$, then define $d_\xi\co C(D(I))\to C(D(J))$ to be comultiplication $\Delta$ on the tensor factor of $\mathcal{A}_2$ corresponding to the component being split and the identity on all other tensor factors of $\mathcal{A}_2$. Each map $d\co C(D(I))\to C(D(J))$ is of bidegree $(1,0)$, i.e. it increases the homological grading by one and preserves the polynomial grading.

Suppose the diagram $D$ has $c_+$ positive crossings and $c_-$ negative crossings. Define $CKh(D) = \bigoplus_{I\in\mathcal{V}(n)} C(D(I))[-c_-]\{c_+-2c_-\}$. The module $CKh(D)=\bigoplus_{i,j\in\mathbb{Z}} CKh^{i,j}(D)$ is bigraded with homological grading $i$ and polynomial grading $j$. Let $CKh^{i,*}(D)$ denote $\bigoplus_{j\in\mathbb{Z}} CKh^{i,j}(D)$. Define $d^{i}\co CKh^{i,*}(D)\to CKh^{i+1,*}(D)$ by $d^{i}=\sum_{|\xi|=i}(-1)^\xi d_\xi$. Then $(CKh(D),d)$ is a chain complex whose homology $Kh(D;R)=\bigoplus_{i,j\in\mathbb{Z}}Kh^{i,j}(D;R)$ is called the {\em Khovanov homology} of $D$ with coefficients in $R$. If $R=\mathbb{Z}$, then we will write $Kh(D)$ in place of $Kh(D;\mathbb{Z})$.

Alternately, one can think of $CKh(D)$ as a bigraded module freely generated by enhanced states. An {\em enhanced state} of $D$ is a Kauffman state of $D$ where each component is either labeled with a $1$ or an $x$. The differential can be described using incidence numbers. Let $s$ be an enhanced state viewed as a basis element of $CKh^{i,j}(D)$, and suppose that $d^i(s) = \sum \pm t_k$ where each $t_k$ is an enhanced state viewed as a basis element of $CKh^{i+1,j}(D)$. The incidence number of two enhanced states $s$ and $t$ relative to $d$, denoted $d(s\co t)$, is defined to be zero unless $s$ and $t$ are related in one of the following ways, in which case $d(s\co t)=1$.
\begin{enumerate}
\item The state $t$ can be obtained from $s$ by merging two components $\gamma_1$ and $\gamma_2$ of $s$ into one component $\gamma$ of $t$. Moreover, if $\gamma_1$ and $\gamma_2$ are both labeled $1$, then $\gamma$ is also labeled $1$, and if one of $\gamma_1$ and $\gamma_2$ is labeled $1$ and the other is labeled $x$, then $\gamma$ is labeled $x$. 
\item The state $t$ can be obtained by splitting one component $\gamma$ of $s$ into two components $\gamma_1$ and $\gamma_2$ of $t$. Moreover, if the label on $\gamma$ is $1$, then one of $\gamma_1$ and $\gamma_2$ is labeled $1$ and one is labeled $x$. If the label on $\gamma$ is $x$, then both $\gamma_1$ and $\gamma_2$ are labeled $x$.
\end{enumerate}
Then $d^i(s) = \sum \pm d(s\co t) t$ where the sum is taken over all enhanced states $t$ and the sign is determined by $(-1)^\xi$ for $\xi$ being the edge in the hypercube between the underlying Kauffman states of $s$ and $t$.

See Figures \ref{figure:12n888} and \ref{figure:Kh12n888} for a diagram of the mirror of the knot $12n_{888}$ and a table of its Khovanov homology $Kh(\overline{12n_{888}})$.

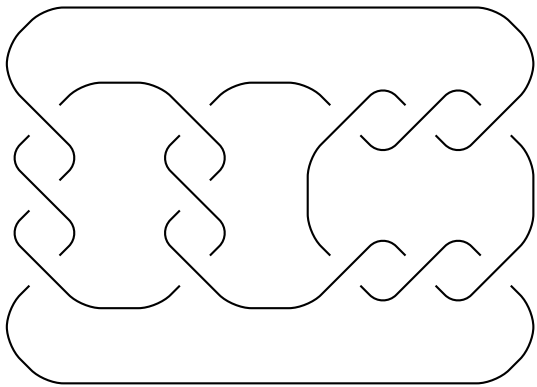
\begin{figure}[h]
$$\begin{tikzpicture}[thick, rounded corners = 2.5mm]

\draw (6.7,2.3) -- (7,2)--(8,3)--(8,3.5)--(7.5,4)-- (1.5,4) -- (1,3.5) -- (1,3) -- (2,2) -- (1.7,1.7);
\draw (1.3,1.3) -- (1,1) -- (2,0) -- (3,0) -- (3.3,.3);
\draw (3.7,.7) -- (4,1) -- (3,2) -- (3.3,2.3);
\draw (1.3,2.3) -- (1,2) -- (2,1) -- (1.7,.7);
\draw (1.7,2.7) -- (2,3) -- (3,3) -- (4,2) -- (3.7,1.7);
\draw (3.3,1.3) -- (3,1) -- (4,0) -- (5,0) -- (6,1) -- (6.3,.7);

\draw (3.7,2.7) -- (4,3) -- (5,3) -- (5.3,2.7);
\draw (5.7,2.3) -- (6,2) -- (7,3) -- (7.3,2.7);
\draw (5.3,.7) -- (5,1) -- (5,2) -- (6,3) -- (6.3,2.7);
\draw (5.7,.3) -- (6,0) -- (7,1) -- (7.3,.7);
\draw (6.7,.3) -- (7,0) --(8,1) -- (8,2) -- (7.7,2.3);
\draw (7.7, .3) -- (8,0) -- (8,-.5) -- (7.5,-1) -- (1.5,-1) -- (1,-.5) -- (1,0) -- (1.3,.3);

\end{tikzpicture}$$
\caption{The mirror of the knot $12n_{888}$.}
\label{figure:12n888}
\end{figure}

\begin{figure}[h]
\begin{tabular}{|| c || c | c | c | c | c | c | c | c | c | c | c | c | c ||}
\hline
\hline
$j\backslash i$ & -12 & -11 & -10  & -9 & -8 & -7 & -6 & -5 & -4 & -3 & -2 & -1 & 0 \\
 \hline

 \hline
-9 & & & & & & & & & & & & & 1 \\
 \hline
 -11 & & & & & & & & & & & & &1 \\
 \hline
 -13 & & & & & & & & & & & 1& & \\
 \hline
  -15  & & & & & & & & & 2 & &$1_2$ & & \\
 \hline
   -17  & & & & & & & &  3 & 1,$1_2$ & 1 & & &  \\
 \hline
    -19  & & & & & & &4 & 2,$3_2$ & & & & & \\
 \hline
    -21   & & & & & & 4 & 3,$4_2$ & 1 & & & & & \\
 \hline
   -23     & & & & & 5& 4,$4_2$ &  & & & & & & \\
 \hline
      -25   & & & & 3& 4,$5_2$ & & & & & & & & \\
 \hline
       -27   & & &3  & 5,$3_2$ &  & & & & & & & & \\
 \hline
       -29   & & 1 &3,$3_2$ & & & & & & & & &  &\\
 \hline
      -31      &  &3,$1_2$ & & & & & & & & & & & \\
 \hline
         -33    & 1 & & & & & & & & & & & & \\
 \hline
 \hline
\end{tabular}
\caption{The Khovanov homology $Kh(\overline{12n_{888}})$ of the mirror of the knot  $12n_{888}$. A constant $k$ denotes a summand $\mathbb{Z}^k$, and the notation $k_{\ell}$ denotes a summand $\mathbb{Z}_{\ell}^k$. Computations done by \cite{BNG:JavaKh}.}
\label{figure:Kh12n888}
\end{figure}

\subsection{Chromatic homology}
The {\em chromatic polynomial} $P_G(\lambda)$ of a graph $G$ is the unique polynomial such that $G$ has $P_G(\lambda)$ proper $\lambda$-colorings of its vertices. Let $s$ be a spanning subgraph of $G$, let $|s|$ be the number of edges in $s$, and let $k(s)$ be the number of components of $s$. Then 
$$P_G(\lambda) = \sum_{s} (-1)^{|s|}\lambda^{k(s)}.$$
The chromatic homology $H(G)$ of $G$ categorifies the polynomial 
$$P_G(1+x) = \sum_{s} (-1)^{|s|}(1+x)^{k(s)},$$
and is constructed as follows.

Suppose the edge set of $G$ is $E(G)=\{e_1,\dots,e_n\}$. The set $S(G)$ of spanning subgraphs of $G$ is in one-to-one correspondence with $\mathcal{V}(n)$. Define $G(I)$ to be the spanning subgraph of $G$ with edge set $E(G(I))=\{e_i\in E(G)~|~k_i=1\}$.

Again, we let $\mathcal{A}_2=R[x]/(x^2)$, where now $1$ has bigrading $(0,0)$ and $x$ has bigrading $(0,1)$. Define $C(G(I)) = \mathcal{A}_2^{\otimes k(G(I))}[h(I)]$, where $k(G(I))$ is the number of components of $G(I)$. Let $m$ be the multiplication map defined in Equation \ref{eq:edgemaps}. Let $\xi$ be an edge in the hypercube $\{0,1\}^n$ from vertex $I$ to vertex $J$. If $k(G(I)) = k(G(J))+1$, then define the map $d_{\xi}\co  C(G(I))\to C(G(J))$ to be multiplication on the factors of $\mathcal{A}_2$ that correspond to the components being merged and the identity on all other factors of $\mathcal{A}_2$. If $k(G(I))=k(G(J))$, then define $d_{\xi}\co C(G(I))\to C(G(J))$ to be the identity map.

Define $C(G) = \bigoplus_{I\in\mathcal{V}(n)} C(G(I))$. Let $C^{i,j}(G)$ be the summand of $C(G)$ in homological grading $i$ and polynomial grading $j$. Define $C^{i,*}(G) = \bigoplus_{j\in\mathbb{Z}} C^{i,j}(G)$ to be the summand in homological grading $i$. The differential $d^i\co C^{i,*}(G)\to C^{i+1,*}(G)$ is defined by $d^i = \sum_{h(\xi)=i} (-1)^{\xi} d_{\xi}$. Then $(C(G),d)$ is a chain complex whose homology $H(G;R)=\bigoplus H^{i,j}(G;R)$ is the {\em chromatic homology} over $R$. If $R=\mathbb{Z}$, then we use $H(G)$ in place of $H(G;\mathbb{Z})$. Figure \ref{figure:chromaticexample} shows an example of a graph and its chromatic homology.
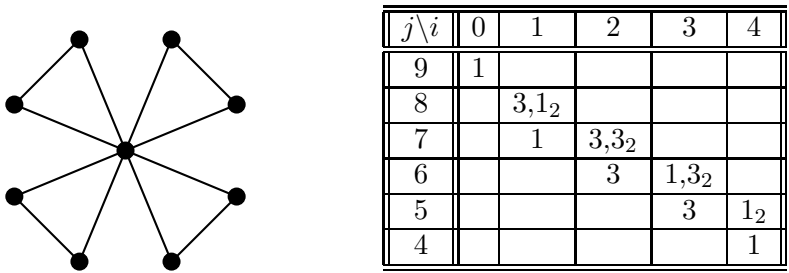
\begin{figure}[h]
\begin{minipage}[b]{0.40\linewidth}
    \centering
$\begin{tikzpicture}[thick,scale = .8]

\fill(1.847,.765) circle (.15cm);
\fill (.765,1.847) circle (.15cm);
\fill (-.765, 1.847) circle (.15cm);
\fill(-1.847,.765) circle (.15cm);
\fill (-1.847,-.765) circle (.15cm);
\fill (-.765,-1.847) circle (.15cm);
\fill (.765,-1.847) circle (.15cm);
\fill (1.847, -.765) circle (.15cm);

\draw (0,0) -- (1.847,.765) -- (.765,1.847) -- (0,0);
\draw (0,0) -- (-1.847,.765) -- (-.765,1.847) -- (0,0);
\draw (0,0) -- (1.847,-.765) -- (.765,-1.847) -- (0,0);
\draw (0,0) -- (-1.847,-.765) -- (-.765,-1.847) -- (0,0);

\fill (0,0) circle (.15cm);

\end{tikzpicture}$
\end{minipage}
 \begin{minipage}[b]{0.30\linewidth}
    \centering

\begin{tabular}{ || c || c | c | c | c | c ||}
\hline\hline
$j \backslash i$ & 0 & 1 & 2 & 3 & 4 \\
\hline
\hline
9 & 1 & & & &  \\
\hline
8 &  & 3,$1_2$ & & & \\
\hline
7 & & 1 & 3,$3_2$ & & \\
\hline
6 & & & 3 & 1,$3_2$ &  \\
\hline
5 & & & & 3 & $1_2$ \\
\hline
4 & & & & & 1 \\
\hline\hline
\end{tabular}
\par\vspace{0pt}
\end{minipage}
\caption{The graph obtained by gluing four triangles at a vertex has the displayed chromatic polynomial cohomology.}
\label{figure:chromaticexample}
\end{figure}

An {\em enhanced state} of $G$ is a spanning subgraph $H$ of $G$ where each component of $H$ is labeled either $1$ or $x$. The incidence between two enhanced states $s$ and $t$, denoted $d(s\co t)$, is either zero or one. The incidence number is one if and only if $t$ can be obtained from $s$ by adding an edge $e$ and one of the following conditions hold.
\begin{enumerate}
\item Adding the edge $e$ to $s$ merges two components $\gamma_1$ and $\gamma_2$ of $s$ into one component $\gamma$ of $t$. If $\gamma_1$ and $\gamma_2$ are labeled $1$, then $\gamma_2$ is labeled $1$, and if one of $\gamma_1$ and $\gamma_2$ is labeled $1$ and the other is labeled $x$, then $\gamma$ is labeled $x$. All other labels are the same in $s$ and $t$.
\item Adding the edge $e$ to $s$ preserves the number of components. Then there is a natural bijection between the components of $s$ and $t$, and that bijection must preserve labels.
\end{enumerate}
As in Khovanov homology, the differential $d$ can be written as $d(s) = \sum \pm d(s\co t) t$ where $t$ ranges over all other enhanced states. 

Helme-Guizon, Przytycki, and Rong \cite{HPR:Torsion} prove that the chromatic homology $H(G)$ is entirely supported on two adjacent $(i,j)$ diagonals and that all of the torsion is supported on the upper diagonal.
\begin{theorem}[Helme-Guizon, Przytycki, Rong]
\label{theorem:ThinChromatic}
Let $G$ be a graph with $n$ vertices. Then $H^{i,j}(G)=0$ unless $n-1\leq i+j \leq n$ and $\Tor H^{i,j}(G)=0$ unless $i+j=n$.
\end{theorem}

Helme-Guizon and Rong \cite{HR:Arbitrary} also give a categorification of the chromatic polynomial using an arbitrary algebra $\mathcal{A}$ in place of $\mathcal{A}_2$. The spaces in this categorification are tensor products of $\mathcal{A}$. Instead of using the multiplication in $\mathcal{A}_2$ when adding an edge that merges two components, one instead uses the multiplication from $\mathcal{A}$. It will be useful for us to take this viewpoint, particularly by using a non-standard multiplication on an $R$-module isomorphic to $\mathcal{A}_2$.

\section{Odd torsion in chromatic homology}
\label{section:odd}

In this section, we prove that the chromatic homology of a graph contains no torsion of odd order. The proof is inspired by Shumakovitch's proof that the Khovanov homology of a homologically thin link has no odd torsion \cite{Shumakovitch:Torsion}. Chmutov, Chmutov, and Rong \cite{CCR:Knight} prove that chromatic homology with $\mathbb{Q}$ has the following {\em knight move decomposition}.
\begin{theorem}[Chmutov, Chmutov, Rong]
\label{theorem:CCRKnight}
Let $G$ be a connected graph with $n$ vertices. 
\begin{enumerate}
\item If $G$ is not bipartite, then $H^{i,n-i}(G;\mathbb{Q})\cong H^{i+1,n-i-2}(G;\mathbb{Q})$ for all $i$.
\item If $G$ is bipartite, then 
$$
H^{0,n}(G;\mathbb{Q}) \cong  H^{1,n-2}(G;\mathbb{Q})\oplus \mathbb{Q}~\text{and}~
H^{i,n-i}(G;\mathbb{Q})\cong  H^{i+1,n-i-2}(G;\mathbb{Q})$$
for all $i>0$.
\end{enumerate}
\end{theorem}
A {\em knight move pair} (depicted in Figure \ref{figure:tetromino}) is a pair of summands of $\mathbb{Q}$ that differ in bigrading by $(1,-2)$. In Lemma \ref{lemma:knight}, we show that chromatic homology over $\mathbb{Z}_p$ for odd prime $p$ satisfies an analogous knight move decomposition. Using the two knight move decompositions, we show that  $\dim H(G;\mathbb{Q}) = \dim H(G;\mathbb{Z}_p)$, and hence $H(G)$ has no odd torsion.

The knight move decomposition of Lemma \ref{lemma:knight} follows from a spectral sequence construction on chromatic homology. Specifically, there is a new differential $d_L$ that anti-commutes with the usual chromatic differential $d$. The chromatic complex of a connected graph $G$ with differential $d + d_L$ is a filtered complex whose homology is either trivial or dimension two (see Proposition \ref{proposition:oddLee}). The spectral sequence associated to the filtered complex converges at the second page for grading reasons, and consequently chromatic homology with $\mathbb{Z}_p$ coefficients has a knight move decomposition.

Our first task is to define the differential $d_L$ and prove basic properties of the chromatic complex with differential $d + d_L$. Let $p$ be an odd prime, and let $R=\mathbb{Z}_p$ be the integers modulo $p$. Following Lee \cite{Lee:Endo} and Chmutov, Chmutov, and Rong \cite{CCR:Knight}, define a new edge map as follows. Let $m_L\co \mathcal{A}_2\otimes \mathcal{A}_2\to\mathcal{A}_2$ be defined by $m_L(1\otimes 1) = m_L(x\otimes 1) = m_L(1 \otimes x) = 0$ and $m_L(x\otimes x) = 1$. Let $\xi$ be an edge from vertex $I$ to vertex $J$. If $k(G(I)) = k(G(J))+1$, then define $d_{L,\xi}\co C(G(I)) \to C(G(J))$ to be $m_L$ on the factors of $\mathcal{A}_2$ corresponding to the components being merged and the identity on all other factors of $\mathcal{A}_2$. If $k(G(I)) = k(G(J))$, then define $d_{L,\xi}$ to be the zero map. The map $d_L^i\co C^{i,*}(G)\to C^{i+1,*}(G)$ is defined by $d_L^i = \sum_{h(\xi)=i} (-1)^{\xi} d_{L,\xi}$. Note that $d_L$ does not preserve the polynomial grading.

Both $m_L$ and $m+m_L$ are multiplications on $\mathcal{A}_2$, and these two algebra structures on $\mathcal{A}_2$ yield differentials $d_L$ and $d+d_L$. Therefore \cite{HR:Arbitrary} implies that the homologies of the complexes $(C(G),d)$ and $(C(G),d+d_L)$ are independent of the ordering of the edges of $G$. Since $(d+d_L)\circ (d+d_L)=0$, it follows that $d_L\circ d + d \circ d_L=0$. Hence $d_L$ induces a map on the $H(G;\mathbb{Z}_p)$. Denote the homology of $(C(G), d+d_L)$ by $H_{d+d_L}(G;\mathbb{Z}_p)$. 

Recall that if $\mathcal{A}$ is an arbitrary algebra, then the complex and homology using the multiplication from $\mathcal{A}$ are denoted by $C(G;\mathcal{A})$ and $H(G;\mathcal{A})$ respectively. Since the multiplication on $\mathcal{A}$ may or may not preserve polynomial grading, we will consider $H(G;\mathcal{A})$ as a singly graded $R$-module. Let $e$ be an edge of $G$, and let $G-e$ and $G/e$ denote the graphs obtained from $G$ by deleting and contracting the edge $e$ respectively. 

\begin{theorem}[Helme-Guizon, Rong]
\label{theorem:LES}
Let $G$ be a graph, and let $e$ be an edge of $G$. There is a short exact sequence of complexes
$$0\to C^{i-1}(G/e;\mathcal{A})\xrightarrow{\alpha} C^i(G;\mathcal{A})\xrightarrow{\beta} C^i(G-e;\mathcal{A})\to 0$$
that yields a long exact sequence of homology groups
$$\cdots  \to H^{i-1}(G/e;\mathcal{A}) \xrightarrow{\alpha^*} H^i(G; \mathcal{A}) \xrightarrow{\beta^*} H^i(G-e;\mathcal{A}) 
\xrightarrow{\partial} H^i(G/e;\mathcal{A})\to\cdots.$$
\end{theorem}
The connecting homomorphism $\partial$ has the following description. Adding the edge $e$ to $G-e$ and then contracting it to obtain $G/e$ either induces a natural bijection on the components or merges two components and induces a natural bijection on the remaining components. In the first case $\partial$ is multiplication $(-1)^i$ and in the second case $\partial$ is given by $(-1)^i$ times  multiplication in $\mathcal{A}$ of the labels on the components being merged.

The homology $H_{d+d_L}(G;\mathbb{Z}_p)$ can be computed explicitly. In Section \ref{section:even}, we will use a similar argument to compute another homology, and so we state the following result with seemingly more general than necessary language.

Let $R$ be a commutative ring, let $\mathcal{A}$ be a free rank-two $R$-module with basis $\{a_0,a_1\}$ with multiplicative identity. Suppose that the multiplication $m\co \mathcal{A}\otimes \mathcal{A} \to \mathcal{A}$ is given by $m(a_0\otimes a_1) = m(a_1\otimes a_0) = 0$ and $m(a_i\otimes a_i) = \pm a_i$ for $i=0,1$. Suppose $G$ is bipartite with vertex bipartiion $A_0\sqcup A_1$. Consider the spanning subgraph of $G$ with no edges. Define two enhanced states $S_0$ and $S_1$ using the subgraph with no edges by assigning labels as follows. In $S_0$, each vertex in $A_0$ is labeled $a_0$ and each vertex in $A_1$ is labeled $a_1$. In $S_1$, each vertex in $A_0$ is labeled $a_1$ and each vertex in $A_1$ is labeled $a_0$. Then $S_0$ and $S_1$ are cycles in $H(G;\mathcal{A})$. 

\begin{lemma}
\label{lemma:meta}
Let $R$ and $\mathcal{A}$ be as in the above paragraph. Suppose that $G$ is a connected graph.
\begin{enumerate}
\item If $G$ is not bipartite, then $H^{i}(G;\mathcal{A})=0$ for all $i$.
\item If $G$ is bipartite, then $H^{i}(G;\mathcal{A}) = 0$ for all $i>0$ and $H^{0}(G;\mathcal{A}) \cong R\oplus R$ with basis represented by $\{S_0,S_1\}$.
\end{enumerate}
\end{lemma}
\begin{proof}
Helme-Guizon and Rong \cite{HR:Arbitrary} show that given a rank two algebra $\mathcal{A}$ with the above properties, if $G$ is a tree, then $H^0(G;\mathcal{A}) = R\oplus R$ and $H^i(G;\mathcal{A}) = 0$ for $i>0$. It is easy to check that both $S_0$ and $S_1$ are cycles, and thus represent a basis for $H^0(G;\mathcal{A})$. 

We complete the proof by induction on $m$, the number of edges in $G$. Since the result is known if $G$ is a tree, we can suppose that $G$ has an edge $e$ that is not a bridge. Let $G-e$ and $G/e$ be the deletion and contraction of $e$ from $G$ respectively. By \cite[Lemma~3.5]{CCR:Knight}, there are three cases to consider:
\begin{enumerate}
\item all three graphs $G$, $G-e$, and $G/e$ are not bipartite,
\item the graphs $G$ and $G-e$ are bipartite, but $G/e$ is not, and
\item the graphs $G-e$ and $G/e$ are bipartite, but $G$ is not.
\end{enumerate}

In case (1), the inductive hypothesis implies that both $H(G-e;\mathcal{A})$ and $H(G/e;\mathcal{A})$ are zero, and so the long exact sequence of Theorem \ref{theorem:LES} implies that $H(G;\mathcal{A})=0$ as well. In case (2), the inductive hypothesis implies that $H(G/e;\mathcal{A})=0$ and that $H^0(G-e;\mathcal{A}) = R \oplus R$ while $H^i(G-e;\mathcal{A})=0$ for $i>0$. Therefore, the long exact sequence of Theorem \ref{theorem:LES} implies that $H^0(G;\mathcal{A})=R \oplus R$ and $H^i(G;\mathcal{A})=0$ for $i>0$.

In case (3), the inductive hypothesis implies that both $H^0(G-e;\mathcal{A})$ and $H^0(G/e;\mathcal{A})$ are $R\oplus R$, while $H^i(G-e;\mathcal{A}) = H^i(G/e;\mathcal{A})=0$ for $i>0$. The long exact sequence of Theorem \ref{theorem:LES} becomes
$$0\to H^0(G;\mathcal{A}) \to H^0(G-e;\mathcal{A}) \xrightarrow{\partial} H^0(G/e;\mathcal{A}) \to H^1(G;\mathcal{A}) \to 0$$
where all unlisted groups are zero. If $\partial$ is an isomorphism, then $H^0(G;\mathcal{A})=H^1(G;\mathcal{A})=0$ as desired. Since both $G-e$ and $G/ e$ are bipartite and $G$ is not, both vertices incident to $e$ lie in the same set of the bipartition. Define $S_i(G-e)$ and $S_i(G/e)$ to be the state $S_i$ for $i=1,2$ and for the graphs $G-e$ and $G/e$ respectively. Then $\partial(S_0(G-e))=\pm S_0(G/e)$ and $\partial(S_1(G-e)) =\pm S_1(G/e)$, and hence $\partial$ is an isomorphism.
\end{proof}

\begin{proposition}
\label{proposition:oddLee}
Let $G$ be a connected graph, and let $p$ be an odd prime.
\begin{enumerate}
\item If $G$ is not bipartite, then $H^{i}_{d+d_L}(G;\mathbb{Z}_p)=0$ for all $i$.
\item If $G$ is bipartite, then $H^{i}_{d+d_L}(G;\mathbb{Z}_p) = 0$ for all $i>0$, and $$H^{0}_{d+d_L}(G;\mathbb{Z}_p) \cong \mathbb{Z}_p\oplus \mathbb{Z}_p.$$
\end{enumerate}
\end{proposition}
\begin{proof}
Let $\mathcal{A}$ be the $\mathbb{Z}_p$-module with basis $\{1,x\}$ and multiplication $m+m_L$. Then $1$ is the multiplicative identity of $\mathcal{A}$. Let $a_0=\frac{1}{2}(x+1)$ and $a_1=\frac{1}{2}(x-1)$. We have $(m+m_L)(a_0\otimes a_1) = (m+m_L)(a_1\otimes a_0) = 0$, $(m+m_L)(a_0\otimes a_0) = a_0$, and $(m+m_L)(a_1\otimes a_1)=-a_1$. Hence the algebra $\mathcal{A}$ satisfies the conditions of Lemma \ref{lemma:meta}, and the result follows.
\end{proof}

Proposition \ref{proposition:oddLee} shows that the homology of $H_{d+d_L}(G;\mathbb{Z}_p)$ is either rank zero or rank two. The differential $d+d_L$ is non-increasing with respect to the polynomial grading, which induces a filtration on the complex $C(G;\mathbb{Z}_p)$ as described below. The filtered structure of this chain complex yields a spectral sequence whose homology is $H_{d+d_L}(G;\mathbb{Z}_p)$. The behavior of this spectral sequence is encapsulated in Lemma \ref{lemma:knight}.

Let $C$ be an $R$-module, and let $(C,d)$ be a chain complex. Suppose that there exists submodules $\mathcal{F}^i C$ of $C$ for each $i\in \mathbb{Z}$ such that $\mathcal{F}^i C \subseteq \mathcal{F}^{i+1}(C)$, $\bigcup_{i\in Z}\mathcal{F}^i C = C$, and  $d(\mathcal{F}^i(C))\subseteq \mathcal{F}^i(C)$. Then $(C,d)$ is a {\em filtered chain complex}. Each filtered chain complex has an associated spectral sequence that converges to the homology $H(C,d)$. Details of the construction can be found in McCleary \cite{M:Spectral}. Recall that a map on a bigraded complex that increases homological grading by $k$ and increases polynomial grading by $\ell$ is said to be of {\em bidegree} $(k,\ell)$. 

The maps $d$ and $d_L$ anti-commute on $C(G;\mathbb{Z}_p)$, and thus $d_L$ induces a map $d_L^*$ on the homology $H(G;\mathbb{Z}_p)$ of the complex $(C(G;\mathbb{Z}_p),d)$. The following theorem implies that the summands of $H(G;\mathbb{Z}_p)$ can be arranged in pairs whose bigradings differ by $(1,-2)$ except if $G$ is bipartite, then there will be two summands of $\mathbb{Z}_p$ in $H^0(G;\mathbb{Z}_p)$ that are not contained in any pair. Compare this result to Theorem \ref{theorem:CCRKnight}.
\begin{lemma}
\label{lemma:knight}
Let $G$ be a connected graph with $n$ vertices, and let $p$ be an odd prime. 
\begin{enumerate}
\item If $G$ is not bipartite, then $d_L^*\co H^{i,n-i}(G;\mathbb{Z}_p)\to H^{i+1,n-i-2}(G;\mathbb{Z}_p)$ is an isomorphism for all $i$.
\item If $G$ is bipartite, then $d_L^*\co H^{i,n-i}(G;\mathbb{Z}_p)\to H^{i+1,n-i-2}(G;\mathbb{Z}_p)$ is an isomorphism for all $i\geq 1$ and has one dimensional kernel when $i=0$.
\end{enumerate}
\end{lemma}
\begin{proof}
Theorem \ref{theorem:ThinChromatic} states that $H(G)$ is entirely supported on two diagonals, i.e. in bigradings $(i,j)$ such that $i+j=n$ or $n-1$ where $n$ is the number of vertices of $G$. Additionally, the torsion of $H(G)$ is supported on bigradings $(i,j)$ where $i+j=n$. The universal coefficient theorem then implies that $H(G;\mathbb{Z}_p)$ is entirely supported in bigradings $(i,j)$ where $i+j=n$ or $n-1$.

The map $d$ is of bidegree $(1,0)$, while the map $d_L$ is of bidegree $(1,-2)$. Then $(C(G), d+d_L)$ has a filtration given by half the polynomial grading. The $E_1$ page of the associated spectral sequence is the chromatic homology $H(G;\mathbb{Z}_p)$ with $\mathbb{Z}_p$ coefficients and the $E_\infty$ page is $H_{d+d_L}(G;\mathbb{F}_p)$. The differential on the $E_1$ page is $d_L^*$, i.e. the complex $(H(G;\mathbb{Z}_p),d_L^*)$ is the $E_1$ page of the spectral sequence. Moreover, the bidegree for the map on the $E_r$ page is $(1,-2r)$. Therefore, all differentials past $d_L^*$ in the spectral sequence are zero, and thus $E_2=E_\infty$. Proposition \ref{proposition:oddLee} implies the result.
\end{proof}

Our proof that the chromatic homology of a graph can only have torsion of order two comes in two parts. The first part is the following result, which implies that all torsion in chromatic homology is of order $2^k$ for some $k$.

\begin{theorem}
\label{theorem:OddTorsion}
The chromatic homology $H(G)$ of the graph $G$ contains no torsion of odd order.
\end{theorem}
\begin{proof}
Suppose that $H(G)$ contains $p^k$-torsion for some odd prime $p$ and some $k>0$. Let $H^{i,j}(G)$ be the summand with minimum homological grading $i$ containing $p^k$-torsion. Since all torsion appears in a summand $H^{i,j}(G)$ where $i+j=n$, it follows that the corresponding polynomial grading is $j=n-i$. Let $\dim_{\mathbb{Z}_p}(i,j)$ and $\dim_{\mathbb{Q}}(i,j)$ denote $\dim H^{i,j}(G;\mathbb{Z}_p)$ and $\dim H^{i,j}(G;\mathbb{Q})$ respectively. The universal coefficient theorem implies that $\dim_{\mathbb{Z}_p}(i-1,n-i) \geq \dim_{\mathbb{Q}}(i-1,n-i)$.
By Lemma \ref{lemma:knight}, 
$$\dim_{\mathbb{Z}_p}(i-1,n-i) = \begin{cases}
\dim_{\mathbb{Z}_p} (i-2, n-i+2) -1 & \text{if $i=2$, $G$ is bipartite,}\\
\dim_{\mathbb{Z}_p} (i-2, n-i+2) &\text{otherwise.}
\end{cases}$$
By Theorem \ref{theorem:CCRKnight}, we have
$$\dim_{\mathbb{Q}} (i-1,n-i) = \begin{cases}
\dim_{\mathbb{Q}}(i-2, n-i+2) -1 & \text{if $i=2$, $G$ is bipartite,}\\
\dim_{\mathbb{Q}} (i-2, n-i+2) &\text{otherwise.}
\end{cases}$$
Therefore $\dim_{\mathbb{Z}_p}(i-2,n-i+2)\geq \dim_{\mathbb{Q}} (i-2,n-i+2)$, and hence $H^{i-2,n-i+2}(G)$ contains $p^k$-torsion for some $k>0$. This contradicts that $i$ is the minimum homological grading where odd torsion appears. Therefore, $H(G)$ contains no torsion of odd order.
\end{proof}

\section{Torsion of order $2^k$ in chromatic homology}
\label{section:even}

In this section, we prove Theorem \ref{theorem:ChromaticTorsion}, i.e. that chromatic homology can only have torsion of order two. Theorem \ref{theorem:OddTorsion} implies that all torsion in chromatic homology is of order $2^k$ for some $k$. Proving that $k=1$ amounts to showing that the Bockstein spectral sequence converges on the correct page. In order to prove our Bockstein convergence result, we show a relationship between the Bockstein differentials and some new maps on chromatic homology. The proof in this section is modeled after a forthcoming paper of Shumakovitch \cite{Shumakovitch:Forthcoming} where he proves that the Khovanov homology of a homologically thin knot can only have torsion of order two.

\subsection{The Bockstein spectral sequence}

In this subsection, we review the construction of the Bockstein spectral sequence. As we will see, it is the exact algebraic tool that we need to show that all torsion of order $2^k$ in $H(G)$ is actually of order $2$.

Let $D$ and $E$ be $R$-modules, and let $f\co D\to D$, $g\co D\to E$, and $h\co E\to D$ be $R$-module homomorphisms such that $\im f = \ker g$, $\im g = \ker h$ and $\im h = \ker f$. The tuple $(D,E,f,g,h)$ is called an {\em exact couple} which we represent by the following diagram.
$$\begin{tikzpicture}
\draw (0,1) node {$D$};
\draw (2,1) node{$D$};
\draw (1,0) node{$E$};
\draw[->] (.3,1) -- (1.7,1);
\draw (1,1) node[above]{$f$};
\draw[->] (1.8,.8) -- (1.2,.2);
\draw (1.5,.5) node[right]{$g$};
\draw[->] (.8,.2) -- (.2,.8);
\draw (.5,.5) node[left]{$h$};
\end{tikzpicture}$$
The map $g\circ h\co  E\to E$ is a differential on $E$ since $(g\circ h) \circ (g \circ h) = g \circ (h \circ g) \circ h$, and $h \circ g =0$. Define $E'= H(E,g\circ h)$, the homology of $E$ with differential $g\circ h$, and define $D' = \im f = \ker g$. Furthermore, define $f'= f|_{D'}\co  D'\to D'$, and define $g'\co D'\to E'$ by $g'(f(y)) = g(y) + (g\circ h)E \in E'$. Finally, define $h'\co E'\to D'$ by $h'(e+(g\circ h)E) = h(e)$. One can check that $(D',E',f',g',h')$ is also an exact couple. Iterating this process yields the {\em spectral sequence associated to the exact couple} $(D,E,f,g,h)$. 

Let $(C,d)$ be a chain complex with integral homology $H(C)$ and mod $p$ homology $H(C;\mathbb{Z}_p)$. Consider the short exact sequence
$$0 \to \mathbb{Z} \xrightarrow{\times p} \mathbb{Z} \xrightarrow{\red p} \mathbb{Z}_p\to 0$$
where $\times p$ is multiplication times $p$ and $\red p$ is reduction modulo $p$. Tensor the complex $(C,d)$ with this short exact sequence to obtain a short exact sequence of complexes. The associated long exact sequence is the exact couple
$$\begin{tikzpicture}
\draw (-.3,1) node {$H(C)$};
\draw (2.3,1) node{$H(C)$};
\draw (1,-.2) node{$H(C;\mathbb{Z}_p)$};
\draw[->] (.3,1) -- (1.7,1);
\draw (1,1) node[above]{$\times p$};
\draw[->] (1.8,.8) -- (1.2,.2);
\draw (1.5,.5) node[right]{$\red p$};
\draw[->] (.8,.2) -- (.2,.8);
\draw (.5,.5) node[left]{$\partial$};
\end{tikzpicture}$$
where $\partial$ is the boundary map in the long exact sequence. The spectral sequence associated to this exact couple is called the {\em $\mathbb{Z}_p$-Bockstein spectral sequence} of $(C,d)$. Some important properties of the Bockstein spectral sequence follow; see McCleary \cite{M:Spectral} for proofs.
\begin{enumerate}
\item The $E_1$ page of the Bockstein spectral sequence is $H(C;\mathbb{Z}_p)$.
\item The $E_\infty$ page of the Bockstein spectral sequence is $H(C)/\Tor H(C)\otimes \mathbb{Z}_p$.
\item If the Bockstein spectral sequence converges at the $E_r$ page, i.e. if $E_r=E_\infty$, then $H(C)$ contains no torsion of order $p^k$ for $k\geq r$.
\end{enumerate}

We consider the $\mathbb{Z}_2$-Bockstein spectral sequence applied to the chromatic complex $(C(G),d)$. Let $y=\sum s_k$ where each $s_k$ is an enhanced state such that $y$ is a cycle in $C(G;\mathbb{Z}_2)$. Denote its $\mathbb{Z}_2$ homology class by $[y]_2$. Define the differential on the $E_1$ page of the $\mathbb{Z}_2$-Bockstein spectral sequence by $\beta = (\partial \; \circ \;\operatorname{red} 2)$. A diagram chase shows that the map $\beta \co H^{i,j}(G;\mathbb{Z}_2)\to H^{i+1,j}(G;\mathbb{Z}_2)$ is defined by $\beta([y]_2)=\left[\frac{1}{2}d(y)\right]_2$. Our goal is thus to show that the homology of the complex $(H(G;\mathbb{Z}_2),\beta)$ is the same as the $E_\infty$ page of the Bockstein spectral sequence, that is $H(G)/\operatorname{Tor} H(G) \otimes \mathbb{Z}_2$. 

Since our path to this result is rather circuitous, we will outline the proof ahead of time. In Lemma \ref{lemma:VerticalIso}, we show there is an isomorphism $\vdown^* \co H^{i,n-i}(G;\mathbb{Z}_2) \to H^{i,n-i-1}(G;\mathbb{Z}_2)$, where $G$ is a graph with $n$ vertices. We use this isomorphism, Theorems \ref{theorem:ThinChromatic}, \ref{theorem:CCRKnight}, and \ref{theorem:OddTorsion}, and the universal coefficient theorem to show that $H(G;\mathbb{Z}_2)$ and $H(G)$ consist of finitely many copies, say $N$, of the two configurations on the middle and right of Figure \ref{figure:tetromino} respectively. If $G$ is bipartite, then $H(G;\mathbb{Z}_2)$ will have two additional summands of $\mathbb{Z}_2$ and $H(G)$ will have two summands of $\mathbb{Z}$ in bigradings $(0,n)$ and $(0,n-1)$. Since there are $N$ torsion summands in $H(G)$, it will suffice to show that the rank of $\beta \co H(G;\mathbb{Z}_2) \to H(G;\mathbb{Z}_2)$ is $N$.

We do not directly analyze the map $\beta$ to prove that it is rank $N$. Instead, we define a map $d_T:C^{i,j}(G;\mathbb{Z}_2)\to C^{i+1,j-1}(G;\mathbb{Z}_2)$ and relate it to $\beta$. The $d_T$ map is inspired by a similar map on Khovanov homology defined by Turner \cite{Turner:KhTwo}. In Proposition \ref{proposition:Turner}, we show that the homology of the chromatic complex over $\mathbb{Z}_2$ with differential $d+d_T$ behaves almost identically to the homology of the chromatic complex over $\mathbb{Z}_p$ or $\mathbb{Q}$ with differential $d+d_L$. Specifically, if $G$ is connected, then $H_{d+d_L}(G;\mathbb{Z}_2)$ is either rank two or rank zero, depending on whether $G$ is bipartite or not. As in Section \ref{section:odd}, the map $d+d_T$ is non-increasing with respect to the polynomial grading, and so a spectral sequence construction implies that $d_T$ induces a map $d_T^*:H^{i,j}(G;\mathbb{Z}_2)\to H^{i+1,j-1}(G;\mathbb{Z}_2)$.

The remainder of the proof consists of three steps: compute the rank of $d_T^*$, relate $d_T^*$ to $\beta$, and finally compute the rank of $\beta$. In order to compute the rank of $d_T^*$, we first show in Lemma \ref{lemma:TurnerCommute} that $d_T^*$ commutes with $\vdown^*$.  Using Proposition \ref{proposition:Turner} and Lemma \ref{lemma:TurnerCommute}, we show in Lemma \ref{lemma:TurnerRank} that the rank of $d_T^*$ is $2N$, where $N$ is the number tetrominoes in $H(G;\mathbb{Z}_2)$. Lemma \ref{lemma:TurnerBockstein} shows that
\begin{equation}
\label{equation:TurnerBockstein1}
d_T^* = \beta \circ \vdown^* + \vdown^* \circ \beta.
\end{equation}
Finally in the proof of Theorem \ref{theorem:ChromaticTorsion}, we show that Equation \ref{equation:TurnerBockstein1} implies that the rank of $\beta$ is $N$. Hence the Bockstein spectral sequence converges at the second page, and the only torsion in the chromatic homology of a graph is of order two.

\subsection{Vertical differentials}
Theorem \ref{theorem:ThinChromatic} implies that the chromatic homology with $\mathbb{Z}_2$ coefficients of a graph $G$ with $n$ vertices is entirely supported on two adjacent diagonals, that is, it is entirely supported in bigradings $(i,j)$ where $i+j=n$ or $n-1$. In this subsection, we show that the chromatic homology on the upper diagonal is isomorphic to the chromatic homology on the lower diagonal. The isomorphism $\vdown^*$ and its inverse $\vup^*$ are the induced maps on chromatic homology coming from two maps defined on the chain complex $C(G;\mathbb{Z}_2)$.

The first of these maps is $\vdown\co  C^{i,j}(G;\mathbb{Z}_2) \to C^{i,j-1}(G;\mathbb{Z}_2)$, defined as follows. Let $s$ be an enhanced state of $G$. Let $t$ be the enhanced state with the same underlying spanning subgraph as $s$ except that exactly one component of $s$ that is labeled $x$ is labeled $1$ in $t$. Then the incidence number of $s$ and $t$ relative to $\vdown$ is $\vdown(s\co t)=1$. For all other states $t'$, we set $\vdown(s\co t')=0$. In other words, $\vdown(s)$ is a sum over all possible ways to change a single component of $s$ that is labeled $x$ to be labeled $1$. 

The second  map $\vup\co C^{i,j}(G;\mathbb{Z}_2)\to C^{i,j+1}(G;\mathbb{Z}_2)$ has a similar definition as $\vdown$, but with one key difference. In order to define $\vup$, we choose a vertex $v_0$ of $G$. If $s$ is an enhanced state of $G$ such that the component of $s$ containing the vertex $v_0$ is labeled $1$, then $\vup(s)$ is defined to be the enhanced state with the same underlying spanning subgraph as $s$ except where the component containing $v_0$ is now labeled $x$. If the component of $s$ containing $v_0$ is labeled $x$, then $\vup(s)=0$.

The reader can check that $\vdown$ and $\vup$ are differentials, i.e. $\vdown\circ\vdown = 0$ and $\vup\circ\vup=0$, and that both $\vdown$ and $\vup$ commute with the chromatic differential $d$ (with $\mathbb{Z}_2$ coefficients). Hence these maps induce maps $\vdown^*$ and $\vup^*$ on chromatic homology $H(G;\mathbb{Z}_2)$ with $\mathbb{Z}_2$ coefficients. As the next lemma states, these maps are isomorphisms.
\begin{lemma}
\label{lemma:VerticalIso}
Let $G$ be a graph with $n$ vertices, and suppose $i+j=n$. The maps $\vdown^*\co H^{i,j}(G;\mathbb{Z}_2) \to H^{i,j-1}(G;\mathbb{Z}_2)$ and $\vup^*\co H^{i,j-1}(G;\mathbb{Z}_2)\to H^{i,j}(G;\mathbb{Z}_2)$ are isomorphisms and inverses of one another.
\end{lemma}
\begin{proof}
We show that $\vdown\circ\vup + \vup\circ\vdown=\id$. Let $s$ be a labeled spanning subgraph of $G$. Suppose the component of $s$ containing $v_0$ is labeled $1$. Then $(\vdown\circ\vup)(s) = s + \sum s_i$ where each $s_i$ is obtained from $s$ by changing the component containing $v_0$ from $1$ to $x$ and by changing another component labeled $x$ to be labeled $1$. Also, $(\vup\circ\vdown)(s) = \sum s_i$ where each $s_i$ is obtained from $s$ by changing the component containing $v_0$ from $1$ to $x$ and by changing another component labeled $x$ to be labeled $1$. Hence $(\vdown\circ\vup + \vup\circ\vdown)(s)=s$, as desired.

Suppose the component of $s$ containing $v_0$ is labeled $x$. Then $(\vdown\circ\vup)(s)=0$ and $(\vup\circ\vdown)(s)=s$. Thus again $(\vdown\circ\vup + \vup\circ\vdown)(s)=s$, as desired. It follows that $\vdown\circ\vup + \vup\circ\vdown=\id$ as a map on the chain groups, and $\vdown^*\circ\vup^* + \vup^*\circ\vdown^*=\id$ as induced maps on chromatic homology $H(G;\mathbb{Z}_2)$. 

If $\alpha\in H(G;\mathbb{Z}_2)$ such that $\vdown^*(\alpha)=0$, then $\vdown^*(\vup^*(\alpha))=\alpha$, and hence $\alpha$ is in the image of $\vdown^*$. Likewise, if $\alpha'\in H(G;\mathbb{Z}_2)$ such that $\vup^*(\alpha')=0$, then $\vup^*(\vdown^*(\alpha'))=\alpha'$, and hence $\alpha'$ is in the image of $\vup^*$. Therefore the homology of the complexes $(H(G;\mathbb{Z}_2),\vdown^*)$ and $(H(G;\mathbb{Z}_2),\vup^*)$ are acyclic. Theorem \ref{theorem:ThinChromatic} implies for each homological grading $i$ there are at most two polynomial gradings $j$ where $H^{i,j}(G;\mathbb{Z}_2)$ are nontrivial. Hence the maps $\vdown^*$ and $\vup^*$ are isomorphisms. 

If $\alpha\in H^{i,n-i}(G;\mathbb{Z}_2)$, then $\vup^*(\alpha)=0$ because $H^{i,n-i+1}(G;\mathbb{Z}_2)=0$. Thus $(\vdown^*\circ \vup^*)(\alpha)=\alpha$. Likewise, if $\alpha\in H^{i,n-i-1}(G;\mathbb{Z}_2)$, then $\vdown^*(\alpha)=0$ because $H^{i,n-i-2}(G;\mathbb{Z}_2)=0$. Thus $(\vup^*\circ \vdown^*)(\alpha)=\alpha$, and hence $\vup^*$ and $\vdown^*$ are inverses of one another.
\end{proof}

Theorem \ref{theorem:CCRKnight} implies that chromatic homology with rational coefficients $H(G;\mathbb{Q})$ can be arranged into knight move pairs except when $G$ is bipartite, there are two additional summands of $\mathbb{Q}$ in $H^{0,*}(G;\mathbb{Q})$. Each summand of $\mathbb{Q}$ in $H(G;\mathbb{Q})$ becomes a $\mathbb{Z}_2$-summand in $H(G;\mathbb{Z}_2)$. Lemma \ref{lemma:VerticalIso} implies that for $i+j=n$, we have $H^{i,j}(G;\mathbb{Z}_2) \cong H^{i,j-1}(G;\mathbb{Z}_2)$. Therefore, each knight move pair in $H(G;\mathbb{Q})$ corresponds to a {\em tetromino} in $H(G;\mathbb{Z}_2)$, i.e. four summands of $\mathbb{Z}_2$ arranged as in Figure \ref{figure:tetromino}. The universal coefficient theorem then implies that each tetromino corresponds to two summands of $\mathbb{Z}$ and one summand of $\mathbb{Z}_{2^k}$ arranged as in Figure \ref{figure:tetromino}.

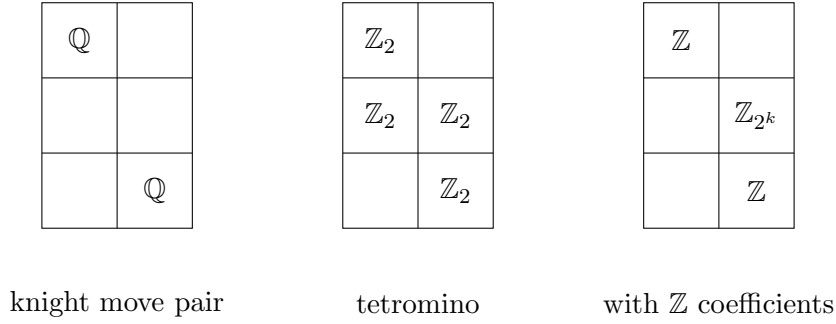
\begin{figure}[h]
$$\begin{tikzpicture}
\draw (0,0) rectangle (2,3);
\draw (0,1) -- (2,1);
\draw (0,2) -- (2,2);
\draw (1,0) -- (1,3);
\draw (0.5,2.5) node{$\mathbb{Q}$};
\draw (1.5,.5) node{$\mathbb{Q}$};
\draw (1,-1) node{knight move pair};

\begin{scope}[xshift = 4cm]
\draw (0,0) rectangle (2,3);
\draw (0,1) -- (2,1);
\draw (0,2) -- (2,2);
\draw (1,0) -- (1,3);
\draw (0.5,2.5) node{$\mathbb{Z}_2$};
\draw (0.5,1.5) node{$\mathbb{Z}_2$};
\draw (1.5,1.5) node{$\mathbb{Z}_2$};
\draw (1.5,.5) node{$\mathbb{Z}_2$};
\draw (1,-1) node{tetromino};
\end{scope}

\begin{scope}[xshift = 8cm]
\draw (0,0) rectangle (2,3);
\draw (0,1) -- (2,1);
\draw (0,2) -- (2,2);
\draw (1,0) -- (1,3);
\draw (0.5,2.5) node{$\mathbb{Z}$};
\draw (1.5,1.5) node{$\mathbb{Z}_{2^k}$};
\draw (1.5,.5) node{$\mathbb{Z}$};
\draw (1,-1) node{with $\mathbb{Z}$ coefficients};
\end{scope}

\end{tikzpicture}$$
\caption{A knight move pair in $H(G;\mathbb{Q})$ corresponds to a tetromino in $H(G;\mathbb{Z}_2)$ and to the configuration on the right in $H(G)$.}
\label{figure:tetromino}
\end{figure}

\subsection{The Turner differential} 

In \cite{Turner:KhTwo} Turner defines a differential on Khovanov homology with $\mathbb{Z}_2$ coefficients. A similar differential on chromatic homology with $\mathbb{Z}_2$ coefficients exists, and we call it the {\em Turner differential}. The Turner differential $d_T\co C^{i,j}(G;\mathbb{Z}_2)\to C^{i+1,j-1}(G;\mathbb{Z}_2)$ is defined as follows. Let $s$ be an enhanced state of $G$. Let $t$ be an enhanced state of $G$ obtained from $s$ by adding an edge that merges two components $\gamma_1$ and $\gamma_2$ of $s$ such that the two components $\gamma_1$ and $\gamma_2$ are labeled $x$ in $s$ and the merged component is labeled $x$ in $t$. Then $d_T(s\co t)=1$. Otherwise $d_T(s\co t)=0$. Said another way, the $d_T$ map has a multiplication that sends $x$ and $x$ to $x$ and all other pairs to zero and has a comultiplication that sends everything to zero. The reader can check that $d_T$ is a differential, $d_T$ commutes with the chromatic differential $d$ (over $\mathbb{Z}_2$), and $(d+d_T) \circ (d + d_T) = 0$. Since $d_T$ commutes with $d$, it induces a map $d_T^* \co H^{i,j}(G;\mathbb{Z}_2) \to H^{i+1,j-1}(G;\mathbb{Z}_2)$.  Note that $d_T$ does not commute with $\vdown$ or $\vup$, and dealing with this unfortunate fact will cause us a bit of work.

Lemma \ref{lemma:meta} implies that the homology of $(C(G;\mathbb{Z}_2), d+d_T)$ has a form similar to the homology of $(C(G;\mathbb{Z}_p),d+d_L)$ for an odd prime $p$.
\begin{proposition}
\label{proposition:Turner}
Let $G$ be a connected graph.
\begin{enumerate}
\item If $G$ is not bipartite, then $H^{i}_{d+d_T}(G;\mathbb{Z}_2)=0$ for all $i$.
\item If $G$ is bipartite, then $H^{i}_{d+d_T}(G;\mathbb{Z}_2) = 0$ for all $i>0$ and 
$$H^{0}_{d+d_T}(G;\mathbb{Z}_2) \cong \mathbb{Z}_2\oplus \mathbb{Z}_2.$$
\end{enumerate}
\end{proposition}
\begin{proof}
Let $\mathcal{A}$ be the algebra whose underlying module structure is $\mathbb{Z}_2[x]/(x^2)$ with multiplication given by $m+m_T$, i.e. $(m+m_T)(1\otimes 1) = 1,~ (m+m_T)(1\otimes x)=(m+m_T)(x\otimes 1) = x,~\text{and}~ (m+m_T)(x\otimes x)=x.$
Define $a_0=x$ and $a_1=x+1$. Then $(m+m_T)(a_0\otimes a_0) = a_0$, $(m+m_T)(a_1\otimes a_1) = a_1$, and $(m+m_T)(a_0\otimes a_1) = (m+m_T)(a_1\otimes a_0)=0$. Therefore, the algebra $\mathcal{A}$ satisfies the conditions of Lemma \ref{lemma:meta}, and the result follows.
\end{proof}

\subsection{Interactions between Bockstein, vertical, and Turner maps}

It remains to compute the rank of $d_T^*$, prove that Equation \ref{equation:TurnerBockstein1} holds, and compute the rank of $\beta$. A key step in the computation of the rank of $d_T^*$ uses the fact that the maps $d_T^*$ and $\vdown^*$ commute.  Since $\vdown^*$ and $\vup^*$ are inverses of one another, it suffices to show that $d_T^*$ commutes with $\vup^*$. Unfortunately, the map $d_T$ does not commute with either $\vdown$ or $\vup$. Despite the technical details, the strategy for proving Lemma \ref{lemma:TurnerCommute} is straightforward: we show that for any cycle $c$ in $(C(G;\mathbb{Z}_2),d)$, we have $(d_T \circ \vup + \vup \circ d_T)(c)$ is in the image of $d$.
\begin{lemma}
\label{lemma:TurnerCommute}
The induced maps $\vdown^*$ and $\vup^*$ on $H(G;\mathbb{Z}_2)$ commute with the induced map $d_T^*$.
\end{lemma}
\begin{proof}

We show that $\vup^*$ and $d_T^*$ commute, and it follows that $\vdown^*$ and $d_T^*$ commute since $\vdown^*$ is the inverse of $\vup^*$. Let $c$ be a cycle in $C^{i,j}(G;\mathbb{Z}_2)$ such that $[c]_2\in H^{i,j}(G;\mathbb{Z}_2)$. Define $\varepsilon = d_T\circ \vup + \vup \circ d_T$. We will show that $\epsilon(c)$ is in the image of $d$, i.e.  that is there exists a $c'\in C^{i,j}(G;\mathbb{Z}_2)$ such that $d(c') = \varepsilon(c)$. Therefore $d_T^*\circ \vup^* + \vup^*\circ d_T^* = 0$, as desired.

In the definition of $\vup$, suppose that marked vertex is $v_0$. Let $c=\sum_{k=1}^m s_k$ where each $s_k$ is an enhanced state and the component containing $v_0$ in each $s_k$ is labeled  $1$ for $1\leq k \leq \ell$ and is labeled $x$ in each $s_k$ for $\ell +1 \leq k \leq m$. Suppose that $1\leq k \leq \ell$. Then $\vup(s_k)$ is an enhanced state with the same underlying subgraph as $s_k$ except the component containing $v_0$ is labeled $x$ instead of $1$. The sum $d_T(\vup(s_k)) = \sum_{r=1}^q t_r$, where each $t_r$ is an enhanced state, can be split into two sums $d_T(\vup(s_k)) = \sum_{r=1}^p t_r + \sum_{r=p+1} ^ q t_r$ such that if $1\leq r \leq p$, then $t_r$ is obtained from $s_k$ by merging the component containing $v_0$ and another component of $\vup(s_k)$ labeled $x$ and if $p+1\leq r \leq q$, then $t_r$ is obtained by merging two components of $\vup(s_k)$ labeled $x$ neither of which contains $v_0$. Each term in $\vup(d_T(s_k))$ is formed by merging two components of $s_k$ that are labeled $x$, then changing the label on the component containing $v_0$ from $1$ to $x$. Hence $\vup(d_T(s_k)) = \sum_{r=p+1}^q t_r$. Thus $\varepsilon(s_k) = \sum_{r=1}^p t_r$. 

Suppose that $\ell+1 \leq k \leq m$. The component containing $v_0$ in $s_k$ is labeled $x$. Hence $\vup(s_k) = 0$. Also, the component containing $v_0$ in each summand of $d_T(s_k)$ is also labeled $x$, and thus $\vup(d_T(s_k))=0$. Hence $\varepsilon (c) =\sum_{k=1}^\ell  \sum_{r=1}^p t_r$. 

Define $c'=\sum_{k=1}^\ell s_k$. We will show that $d(c') = \varepsilon (c)$. Let $t$ be a summand of $\varepsilon(c)$, i.e. let $t$ be an enhanced state of $G$ such that $\varepsilon(s_k\co t)=1$ for some $k$ with $1\leq k \leq \ell$. Then $t$ is obtained from $s_k$ by merging the component containing $v_0$ that is labeled $1$ with another component that is labeled $x$, and labeling the resulting component with $x$. Merging a component labeled $1$ with a component labeled $x$ in $s_k$ and labeling the resulting component in $t$ with $x$ also yields $d(s_k\co t)=1$.

Suppose that $t$ is a summand of $d(c')$, i.e. that $d(s_k'\co t)=1$ for some summand $s_k$ of $c'$. Since the component containing $v_0$ is labeled $1$ for each summand $s_k$ of $c'$, one of two situations arise. In the first case, the component containing $v_0$ in $t$ is labeled $x$ and $t$ is obtained from $s_k$ by merging the component containing $v_0$ with another component labeled $x$. Then, as above, $\varepsilon(s_k\co t)=1$. In the second case, the component containing $v_0$ in $t$ is labeled $1$.  Since $d(c)=0$, it follows that each term in the summand of $d(c)$ where the component containing $v_0$ is labeled $1$ must appear in $d(c)$ an even number of times. All of these contributions must come from a term $s_k$ where $1\leq k \leq \ell$ since if $\ell< k$, then the component containing $v_0$ in $t$ would be labeled $x$. Hence each term where the component containing $v_0$ in $t$ is labeled $1$ in the sum $d(c')$ appears an even number of times. Thus $d(c')=\varepsilon(c)$ as desired.
\end{proof}

The map $d+d_T$ is non-increasing with respect to the polynomial grading, and thus the complex $(C(G;\mathbb{Z}_2),d+d_T)$ is filtered. If $G$ is connected, the $E_\infty$ page of the associated spectral sequence is either rank zero or two. For grading reasons, we know that the $E_3$ page of the sequence is the $E_\infty$ page. We can use Lemma \ref{lemma:TurnerCommute} to show that in fact the $E_2$ page of the sequence is the $E_\infty$ page. Consequently,  $d_T^*$ maps the two $\mathbb{Z}_2$ summands on the left of each tetromino isomorphically onto the two $\mathbb{Z}_2$ summands on the right.

\begin{lemma}
\label{lemma:TurnerRank}
Let $G$ be a connected graph such that $H(G;\mathbb{Z}_2)$ consists of $N$ tetrominoes, and if $G$ is bipartite, two additional summands of $\mathbb{Z}_2$ in $H^0(G;\mathbb{Z}_2)$. Then the rank of $d_T^*\co H(G;\mathbb{Z}_2)\to H(G;\mathbb{Z}_2)$ is $2N$. 
\end{lemma}
\begin{proof}
Recall that the chromatic differential $d$ has bidegree $(1,0)$ and the Turner differential $d_T$ has bidegree $(1,-1)$. The complex $(C(G;\mathbb{Z}_2),d+d_T)$ is filtered by the polynomial grading and gives rise to an associated spectral sequence. The $E_1$ page of the spectral sequence is the chromatic homology $H(G;\mathbb{Z}_2)$ of $G$ with $\mathbb{Z}_2$ coefficients.  If $d_r$ is the map on page $E_r$, then the bidegree of $d_r$ is $(1,-r)$. Because $H(G;\mathbb{Z}_2)=\bigoplus_{i,j}H^{i,j}(G;\mathbb{Z}_2)$ is only supported on two adjacent $(i,j)$-diagonals, the maps $d_r$ are zero for grading reasons when $r> 2$. Proposition \ref{proposition:Turner} implies that the $E_\infty$ page of the spectral sequence is $0$ if $G$ is not bipartite and $\mathbb{Z}_2\oplus \mathbb{Z}_2$ if $G$ is bipartite. 

The differential on the $E_1$ page $H(G;\mathbb{Z}_2)$ is the map induced by Turner's differential. Suppose that $G$ has $n$ vertices. Let $i+j=n$, and let 
\begin{align*}
(d^{i,j}_T)^* & \co H^{i,j}(G;\mathbb{Z}_2) \to H^{i+1,j-1}(G;\mathbb{Z}_2)~\text{and}\\
(d^{i,j-1}_T)^* & \co H^{i,j-1}(G;\mathbb{Z}_2) \to H^{i+1,j-2}(G;\mathbb{Z}_2)
\end{align*}
be the potentially nonzero induced Turner maps. Since $\dim H(G;\mathbb{Z}_2) - \dim H_{d+d_T}(G;\mathbb{Z}_2) = 4N$, it follows that $\rank d_T^*\leq 2N$.

Suppose that $\rank d_T^* < 2N$. Since $d_r=0$ when $r>2$, it follows that the $E_2$ page consists of pairs of summands of $\mathbb{Z}_2$ that differ in bigrading by $(1,-2)$ and if $G$ is bipartite, two additional summands of $\mathbb{Z}_2$ in homological grading $0$. These pairs of summands look like the knight move pairs in Figure \ref{figure:tetromino} except the summands of $\mathbb{Q}$ are replaced with summands of $\mathbb{Z}_2$. Let $(p,q)$ be the bigrading of the $\mathbb{Z}_2$ summand of the $E_2$ page with maximum homological grading. Thus $p+q=n-1$, $\dim E_2^{p,q} = \dim E_2^{p-1,q+2}$, and $E_2^{p,q+1}=0$. For this to occur, one must have $\rank (d_T^{p-1,q+1})^* < \rank (d_T^{p-1,q+2})^*$. However, by Lemma \ref{lemma:TurnerCommute} the following square commutes.
$$\begin{tikzpicture}
\draw (0,2) node{$H^{p-1,q+2}(G;\mathbb{Z}_2)$};
\draw (5,2) node{$H^{p,q+1}(G;\mathbb{Z}_2)$};
\draw (0,0) node{$H^{p-1,q+1}(G;\mathbb{Z}_2)$};
\draw (5,0) node{$H^{p,q}(G;\mathbb{Z}_2)$};

\draw [->] (2,0) -- (3.5,0);
\draw [->] (2,2) -- (3.5,2);
\draw [->] (0,1.6) -- (0,.4);
\draw [->] (5,1.6) -- (5,.4);

\draw (2.75,2) node[above]{$(d_T^{p-1,q+2})^*$};
\draw (2.75,0) node[below]{$(d_T^{p-1,q+1})^*$};
\draw (0,1) node[left]{$(\vdown^{p-1,q+2})^*$};
\draw (5,1) node[right]{$(\vdown^{p,q+1})^*$};

\end{tikzpicture}$$

Lemma \ref{lemma:VerticalIso} implies that both $(\vdown^{p-1,q+2})^*$ and $(\vdown^{p,q+1})^*$ are isomorphisms. Hence it follows that $\rank (d_T^{p-1,q+2})^* = \rank (d_T^{p-1,q+1})^*$,  which is a contradiction. Therefore $\rank d_T^*=2N$, as desired.

 \end{proof}
 
After computing the rank of $d_T^*$, we now express $d_T^*$ in terms of $\beta$ and $\vdown^*$. 
\begin{lemma}
\label{lemma:TurnerBockstein}
Let $G$ be a graph. Then $d_T^*=\beta\circ \vdown^* + \vdown^*\circ\beta$ as maps from $H^{i,j}(G;\mathbb{Z}_2)$ to $H^{i+1,j-1}(G;\mathbb{Z}_2)$.
\end{lemma}
\begin{proof}
Let $c=\sum_{k=1}^\ell s_k$ be a sum of enhanced states $s_k$. One can consider $c$ as a chain in $C(G)$ or in $C(G;\mathbb{Z}_2)$. Suppose that $c$ is a cycle in $C(G;\mathbb{Z}_2)$, and let $[c]_2$ denote the homology class of $c$ in $H(G;\mathbb{Z}_2)$. Recall that $\beta([y]_2) = [\frac{1}{2}d(y)]_2$ where $d$ is the usual Khovanov differential with $\mathbb{Z}$ coefficients. We have 
\begin{align*}
(\beta\circ\vdown^* + \vdown^*\circ\beta)([c]_2) =& \left[\frac{1}{2} \left(\left(d\circ \vdown\right) + \left(\vdown\circ d\right)\right)(c)\right]_2\\
= &\left[\frac{1}{2} \sum_{k=1}^\ell \left(\left(d\circ \vdown\right) + \left(\vdown\circ d\right)\right)(s_k)\right]_2\\
=& \left[\frac{1}{2}\sum_{k=1}^\ell \delta(s_k)\right]_2,
\end{align*} 
where $\delta$ is defined to be $\left(d\circ \vdown\right) + \left(\vdown\circ d\right)$. We will show that 
$$\sum_{k=1}^\ell d_T(s_k) \equiv \frac{1}{2}\sum_{k=1}^\ell \delta(s_k) \mod 2,$$
which implies the result.

The map $\vdown$ does not change the underlying spanning subgraph of an enhanced state, and the maps $d$ and $d_T$ send an enhanced state $s$ to a sum of enhanced states where the underlying spanning subgraph of each enhanced state in the sum is obtained from $s$ by adding an edge. Hence we will consider all enhanced states $t$ whose underlying spanning subgraph can be obtained from one of the underlying spanning subgraph of one of the states $s_k$ by adding a single edge. 

Suppose that $t$ is obtained from $s_k$ by adding an edge that merges two components $\gamma_1$ and $\gamma_2$ into one component $\gamma$ in $t$. Furthermore suppose that $\gamma_1$, $\gamma_2$, and $\gamma$ are all labeled $x$. By definition $d_T(s_k\co t)=1$. For $i=1,2$, let $u_i$ be the enhanced state with the same underlying spanning subgraph as $s_k$ and with all the same labels as $s_k$ except the component $\gamma_i$ is labeled $1$ in $u_i$ instead of $x$. Then $\vdown(s_k\co u_i)=1$ for $i=1,2$. The components $\gamma_1$ and $\gamma_2$ in $u_i$ are labeled $1$ and $x$ in some order. The enhanced state $t$ can be obtained from $u_i$ by adding an edge, and the resulting merged component $\gamma$ is labeled $x$. Since two components labeled $1$ and $x$ are merged to a component labeled $x$, it follows that $d(u_i\co t)=1$ for $i=1$ and $2$. Therefore $(d\circ \vdown)(s_k\co t)=2$. Also, $(\vdown \circ d)(s_k\co t)=0$ since $\gamma_1$ and $\gamma_2$ are both labeled $x$ and merging two components labeled $x$ yields $0$ under the $d$ map. Therefore, $d_T(s_k\co t)  \equiv \frac{1}{2} \delta(s_k\co t)\mod 2$.

Suppose that $t$ is obtained from $s_k$ by adding an edge. If the edge does not merge two components labeled $x$ in $s_k$ into another component labeled $x$ in $t$, then $d_T(s_k\co t)=0$ by definition. Hence for any such state $t$, we must show that 
$\frac{1}{2}\sum_{k=1}^\ell \delta(s_k\co t)$ is even, or equivalently that $\sum_{k=1}^\ell \delta(s_k\co t)$ is divisible by four. There are three cases to consider.
\begin{enumerate}
\item The enhanced state $t$  is obtained by merging two components $\gamma_1$ and $\gamma_2$ in $s_k$ where one of $\gamma_1$ and $\gamma_2$ is labeled $1$ and the other is labeled $x$. The corresponding merged component $\gamma$ in $t$ is labeled $1$.
\item The enhanced state $t$ is obtained from $s_k$ by changing the label on a component $\gamma_1$ of $s_k$ from $x$ to $1$, then merging two other components $\gamma_2$ and $\gamma_3$ of $s_k$ into one component $\gamma$ of $t$. Moreover the label on $\gamma$ is the product (under $m$) of the labels on $\gamma_2$ and $\gamma_3$.
\item The enhanced state $t$ is obtained from $s_k$ by changing the label on a component $\gamma_1$ of $s_k$ from $x$ to $1$, then adding an edge to $s_k$ that does not change the number of components or the labels on any component.
\end{enumerate}

Suppose that $t$ satisfies case (1). Let $u$ be the state whose underlying spanning subgraph is the same as $t$ and that is obtained from $s_k$ by merging $\gamma_1$ and $\gamma_2$ into one component $\gamma$ where $\gamma$ is labeled $x$. Then $d(s_k\co u)=1$ and $\vdown(u\co t)=1$. Let $u'$ be the state whose underlying spanning subgraph is the same as $s_k$ where both $\gamma_1$ and $\gamma_2$ are labeled $1$ and all other components are labeled as in $s_k$. Then $\vdown(s_k\co u')=1$ and $d(u'\co t)=1$. Thus $\delta(s_k\co t)=2$.

Suppose that $t$ satisfies case (2). Let $u$ be the state whose underlying spanning subgraph is the same as $t$, where the component corresponding to $\gamma_1$ is labeled $x$, and the component $\gamma$ is labeled by the product (under $m$) of the labels on $\gamma_2$ and $\gamma_3$. Then $d(s_k\co u)=1$ and $\vdown(u\co t)=1$. Let $u'$ be the state whose underlying spanning subgraph is the same as $s_k$ where $\gamma_1$ is labeled $1$ and all other components are labeled as in $s_k$. Then $\vdown(s_k\co u')=1$ and $d(u'\co t)=1$. Thus $\delta(s_k\co t)=2$.

Suppose that $t$ satisfies case (3). Let $u$ be the state whose underlying spanning subgraph is the same as $t$ and where the component corresponding to $\gamma_1$ is labeled $x$. Then $d(s_k\co u)=1$ and $\vdown(u\co t)=1$. Let $u'$ be the state whose underlying spanning subgraph is the same as $s_k$ where $\gamma_1$ is labeled $1$ and all other components are labeled as in $s_k$. Then $\vdown(s_k\co u')=1$ and $d(u'\co t)=1$. Thus $\delta(s_k\co t)=2$.

In each case, if there exists a $u$ such that $d(s_k\co u)=1$ and $\vdown (u\co t)=1$, then there exists a corresponding $u'$ such that $\vdown(s_k\co u')=1$ and $d(u'\co t)=1$. Moreover, the pair $(s_k,u)$ uniquely determines the pair $(s_k,u')$ and vice versa. Since $c=\sum_{k=1}^\ell$ is a cycle in $C(G;\mathbb{Z}_2)$, it follows that $[d(c)]_2=0$. Therefore, for each enhanced state $u$ there is an even number, say $2m_u$ for some integer $m_u$, of states $s_k$ with $1\leq k \leq \ell$ such that $d(s_k\co u)=1$. Since each pair $(s_k,u)$ yields a pair $(s_k,u')$, it follows that there are also $2m_u$ states $s_k$ with $1\leq k \leq \ell$ such that $\vdown(s_k\co u')=1$. Therefore 
$$\sum_{k=1}^\ell \delta(s_k\co t) = \sum_{\{u~|~  d(s_k\co u)=1\}}4m_u.$$
Thus, for each $t$, the sum $\sum_{k=1}^\ell \delta(s_k\co t)$ is divisible by four, and so 
$$\frac{1}{2}\sum_{k=1}^\ell \delta(s_k\co t) \equiv 0 \mod 2.$$
Hence 
$$\sum_{k=1}^\ell d_T(s_k) \equiv \frac{1}{2}\sum_{k=1}^\ell \delta(s_k) \mod 2,$$
and $d_T^*=\beta\circ \vdown^* + \vdown^*\circ\beta$.
\end{proof}

We now have all of the necessary ingredients to prove that all torsion in the chromatic homology $H(G)$ is of order two.
\begin{proof}[Proof of Theorem \ref{theorem:ChromaticTorsion}]
Theorem \ref{theorem:OddTorsion} states that the chromatic homology $H(G)$ of $G$ has no torsion of odd order. Hence all torsion in $H(G)$ is of order $2^k$ for some positive integer $k$. In order to show that $k$ must equal one (and hence all torsion is of order two), we need to show that the $\mathbb{Z}_2$-Bockstein spectral sequence converges at the second page, that is $E_2=E_\infty$. 

The $E_\infty$ page of the $\mathbb{Z}_2$-Bockstein spectral sequence is $H(G)/\Tor H(G)\otimes \mathbb{Z}_2$. Suppose that $H(G;\mathbb{Z}_2)$ consists of $N$ tetrominoes and if $G$ is bipartite, two additional summands of $\mathbb{Z}_2$. Thus the dimension of $H(G;\mathbb{Z}_2)$ is $4N$ if $G$ is not bipartite or $4N+2$ if $G$ is bipartite. Chmutov, Chmutov, and Rong \cite{CCR:Knight} prove that the dimension (over $\mathbb{Q}$) of $H(G;\mathbb{Q})$ is $2N$ if $G$ is not bipartite or $2N+2$ if $G$ is bipartite. Hence the dimension over $\mathbb{Z}_2$ of the $E_\infty$ page of the $\mathbb{Z}_2$-Bockstein spectral sequence is $2N$ if $G$ is not bipartite or $2N+2$ if $G$ is bipartite. 

Let $(d_T^U)^*$ be the induced Turner map on the upper diagonal of $H(G;\mathbb{Z}_2)$, i.e. the sum of $(d_T^{i,j})^*$ where $i+j=n$, the number of vertices in $G$. Likewise, let $(d_T^L)^*$ be the induced Turner map on the lower diagonal of $H(G;\mathbb{Z}_2)$, i.e. the sum of $(d_T^{i,j})^*$ where $i+j=n-1$.  Since $\vup^*$ is an isomorphism and commutes with $d_T^*$, it follows that $\rank (d_T^U)^* = \rank (d_T^L)^*$, and because $\rank d_T^*= 2N$, it follows that $\rank (d_T^U)^* = \rank (d_T^L)^* = N.$

If we restrict the equality $d_T^*=\beta\circ \vdown^* + \vdown^*\circ\beta$ of Lemma \ref{lemma:TurnerBockstein} to the upper diagonal, then $\vdown^*\circ \beta = 0$ for grading reasons, and hence $(d_T^U)^*=\beta\circ\vdown^*$. Likewise, restricting to the lower diagonal yields $(d_T^L)^* = \vdown^*\circ \beta$. Since $\vdown^*$ is an isomorphism, it follows that the rank of $\beta$ is $N$. Hence the dimension of the $E_2$ page of the Bockstein spectral sequence is $2N$ if $G$ is not bipartite or $2N+2$ if $G$ is bipartite. Thus $E_2=E_\infty$, and therefore the chromatic homology $H(G)$ of $G$ with $\mathbb{Z}$ coefficients has only torsion of order two.
\end{proof}

\begin{proof}[Proof of Theorem \ref{theorem:ChromaticPoly}]
Chmutov, Chmutov, and Rong \cite[Corollary 5.4]{CCR:Knight} prove that the chromatic homology $C(G;\mathbb{Q})$ with rational coefficients is determined by the chromatic polynomial of $G$. Since $\vdown^*$ is an isomorphism of $H(G;\mathbb{Z}_2)$, it follows that each knight move pair in $H(G;\mathbb{Q})$ is replaced with a tetromino in $H(G;\mathbb{Z}_2)$ and the other configuration in Figure \ref{figure:tetromino} in $H(G)$. Theorem \ref{theorem:ChromaticTorsion}  implies that each torsion summand of $H(G)$ in Figure \ref{figure:tetromino} is in fact a $\mathbb{Z}_2$ summand.
\end{proof}

\section{State graphs}
\label{section:mainresults}

In this section, we describe the connection between chromatic homology and Khovanov homology via state graphs, which was first discovered by Helme-Guizon, Przytycki, and Rong \cite{HPR:Torsion}. We use this connection to prove Theorem \ref{theorem:KhTorsion}.

Associated to each Kauffman state $s$ of a link diagram $D$ is a {\em state graph} $G_s(D)$, constructed as follows. The vertices of $G_s(D)$ are in one-to-one correspondence with the components of $s$, and the edges of $G_s(D)$ are in one-to-one correspondence with the traces of the crossings of $D$. An edge of $G_s(D)$ connects two vertices $v_1$ and $v_2$ together if and only if the corresponding trace connects the components of $s$ corresponding to $v_1$ and $v_2$ together. 

The Kauffman state of $D$ where every resolution is an $A$-resolution is called the {\em all-$A$ state} of $D$, and its corresponding state graph, denoted $G_A(D)$, is the {\em all-$A$ state graph} of $D$. Likewise, the state of $D$ where every resolution is a $B$-resolution is called the {\em all-$B$ state} of $D$, and its corresponding state graph, denoted $G_B(D)$, is the {\em all-$B$ state graph} of $D$. See Figure \ref{figure:A-state12n888} for an example of the all-$A$ state of $\overline{12n_{888}}$.

A link diagram $D$ is {\em $A$-adequate} if $G_A(D)$ has no loops and {\em $B$-adequate} if $G_B(D)$ has no loops. A link is {\em adequate} if it has a diagram $D$ that is both $A$-adequate and $B$-adequate and is {\em semi-adequate} if it has a diagram that is either $A$-adequate or $B$-adequate. The {\em girth} of a graph $G$, denoted $\girth(G)$, is the length of the shortest cycle in $G$. If the graph is acyclic, then the girth of the graph is defined to be infinity. If the girth of $G_A(D)$ or $G_B(D)$ is greater than one, then $D$ is $A$-adequate or $B$-adequate respectively.

If $D$ is $A$-adequate, then Theorem \ref{theorem:KhTorsion} implies that the first few homological gradings where $Kh(D)$ is supported have only torsion of order two. Likewise, if $D$ is $B$-adequate, then Theorem \ref{theorem:KhTorsion} implies the the last few homological gradings where $Kh(D)$ is supported have only torsion of order two. Thus if a link is semi-adequate, Theorem \ref{theorem:KhTorsion} tells us something about the torsion in its Khovanov homology. Alternating and Montesinos links are semi-adequate, and semi-adequate knots occur frequently among knots with few crossings. All knots with fewer than twelve crossings except two eleven crossing knots are semi-adequate. Stoimenow \cite{Sto:Semi} observes that at least 249,649 of the 253,293 knots with 15 crossings tabulated in \cite{HT:Table} are semi-adequate.

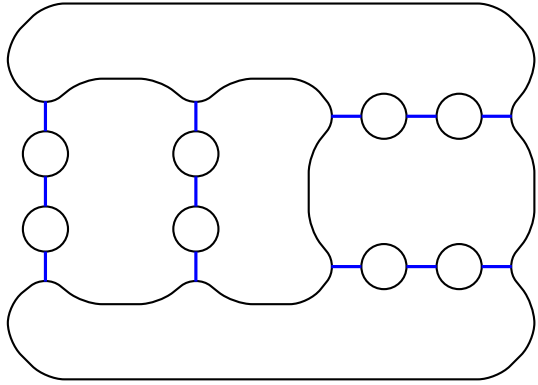
\begin{figure}[h]
$$\begin{tikzpicture}[thick, rounded corners = 2.5mm]

\draw (4,4) -- (1.5,4) -- (1,3.5) -- (1,3) -- (1.5,2.6) -- (2,3) -- (3,3) -- (3.5,2.6) -- (4,3) -- (5,3) -- (5.4,2.5) -- (5,2) -- (5,1) -- 
(5.4,.5) -- (5,0) -- (4,0) -- (3.5,.4) -- (3,0) -- (2,0) -- (1.5,.4) -- (1,0) -- (1,-.5) -- (1.5,-1) -- (7.5,-1) -- (8,-.5) -- (8,0) -- (7.6,.5) --
(8,1) -- (8,2) -- (7.6,2.5) -- (8,3) -- (8,3.5) -- (7.5,4) -- (4,4);

\draw (1.5,2) circle (.3cm);
\draw (1.5,1) circle (.3cm);

\draw (3.5, 2) circle (.3cm);
\draw (3.5,1) circle (.3cm);

\draw (7,.5) circle (.3cm);
\draw (6,.5) circle (.3cm);

\draw (6,2.5) circle (.3cm);
\draw (7,2.5) circle (.3cm);

\draw [very thick, blue] (1.5,2.3) -- (1.5,2.7);
\draw [very thick, blue] (3.5,2.3) -- (3.5,2.7);
\draw [very thick, blue] (1.5,1.7) -- (1.5,1.3);
\draw [very thick, blue] (3.5,1.7) -- (3.5,1.3);
\draw [very thick, blue] (1.5, .7) -- (1.5,.3);
\draw [very thick, blue] (3.5, .7) -- (3.5,.3);

\draw [very thick, blue] (5.3, 2.5) -- (5.7,2.5);
\draw [very thick, blue] (6.3,2.5) -- (6.7,2.5);
\draw [very thick, blue] (7.3,2.5) -- (7.7,2.5);
\draw [very thick, blue] (5.3, .5) -- (5.7,.5);
\draw [very thick, blue] (6.3,.5) -- (6.7,.5);
\draw [very thick, blue] (7.3,.5) -- (7.7,.5);

\end{tikzpicture}$$
\caption{The all-$A$ state of $\overline{12n_{888}}$. Its all-$A$ state graph is four triangles glued along a vertex, as depicted in Figure \ref{figure:chromaticexample}. The girth of $G_A(D)$ is three.}
\label{figure:A-state12n888}
\end{figure}

Helme-Guizon, Przytycki, and Rong \cite{HPR:Torsion} prove the following theorem comparing the chromatic homology of $G_A(D)$ and the Khovanov homology of $D$.
\begin{theorem}
\label{theorem:ChromaticKhovanov}[Helme-Guizon, Przytycki, Rong]
Let $D$ be a diagram with $c_-$ negative crossings and $c_+$ positive crossings. Suppose that $G_A(D)$ has $n$ vertices and girth $g$ with $g>1$. Let $p= i - c_-$ and $q = n-2j + c_+ - 2c_-$.
If $0\leq i < g - 1$ and $j\in\mathbb{Z}$, then there is an isomorphism
$$H^{i,j}(G_A(D)) \cong Kh^{p,q}(D)$$
Moreover, there is an isomorphism of torsion
$$\operatorname{Tor} H^{g-1,j}(G_A(D))\cong \operatorname{Tor} Kh^{g- c_- -1,q}(D)$$
for all $j\in\mathbb{Z}$.
\end{theorem}

Theorem \ref{theorem:ChromaticKhovanov} implies that the Khovanov homology of $\overline{12n_{888}}$ and the chromatic homology of four triangles glued at vertex are isomorphic in the first two homological gradings and have isomorphic torsion in the third homological grading. See Figures \ref{figure:Kh12n888} and \ref{figure:chromaticexample}. Theorem \ref{theorem:ChromaticKhovanov} has an immediate corollary. 

\begin{corollary}
\label{corollary:SameKh}
Let $D$ and $D'$ be link diagrams such that $G_A(D) = G_A(D')$. Suppose that the number of positive and negative crossings in $D$ and $D'$ are $c_{\pm}(D)$ and $c_{\pm}(D')$. Let $g=\girth(G_A(D))$. There is an isomorphism of Khovanov homology
$$Kh^{i,j}(D) \cong Kh^{p,q}(D'),$$
for $-c_-(D)\leq i < -c_-(D) + g - 1$ and all $j$ where $p-c_-(D')= i- c_-(D)$ and $q +c_+(D') - 2c_-(D') = j+c_+(D) - 2c_-(D).$
\end{corollary}

Since the connected sum of four trefoils $3_1\# 3_1 \# 3_1 \# 3_1$ and the mirror of $12n_{888}$ have the same all-$A$ state graph of four triangles glued at a vertex, it follows that their Khovanov homologies agree in the first two nontrivial homological gradings. See Figures \ref{figure:4trefoils} and \ref{figure:KhTref}.
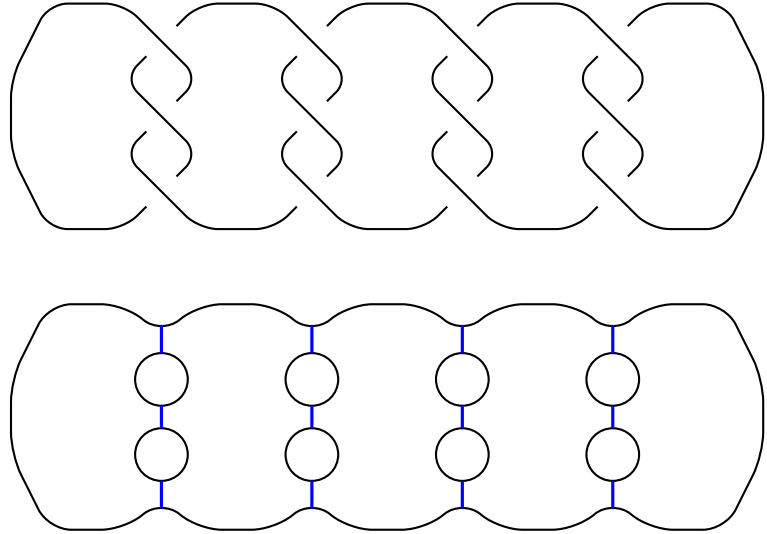
\begin{figure}[h]
$$\begin{tikzpicture}[thick, rounded corners = 2.5mm]

\draw (1.3,.3) -- (1,0) -- (0,0) -- (-.5,1) -- (-.5,2) -- (0,3) -- (1,3) -- (2,2) -- (1.7,1.7);
\draw (1.3,1.3) -- (1,1) -- (2,0) -- (3,0) -- (3.3,.3);
\draw (1.3,2.3) -- (1,2) -- (2,1) -- (1.7,.7);
\draw (1.7,2.7) -- (2,3) -- (3,3) -- (4,2) -- (3.7,1.7);

\begin{scope}[xshift = 2cm]
\draw (1.3,1.3) -- (1,1) -- (2,0) -- (3,0) -- (3.3,.3);
\draw (1.3,2.3) -- (1,2) -- (2,1) -- (1.7,.7);
\draw (1.7,2.7) -- (2,3) -- (3,3) -- (4,2) -- (3.7,1.7);
\end{scope}

\begin{scope}[xshift = 4cm]
\draw (1.3,1.3) -- (1,1) -- (2,0) -- (3,0) -- (3.3,.3);
\draw (1.3,2.3) -- (1,2) -- (2,1) -- (1.7,.7);
\draw (1.7,2.7) -- (2,3) -- (3,3) -- (4,2) -- (3.7,1.7);
\end{scope}

\begin{scope}[xshift = 6cm]
\draw (1.3,2.3) -- (1,2) -- (2,1) -- (1.7,.7);
\draw (1.7,2.7) -- (2,3) -- (3,3) -- (3.5,2) -- (3.5,1) -- (3,0) -- (2,0) -- (1,1) -- (1.3,1.3); 
\end{scope}

\begin{scope}[yshift = -4cm]

\draw[rounded corners = 3mm] (.5,0) -- (0,0) -- (-.5,1) -- (-.5,2) -- (0,3) -- (1,3) -- (1.5,2.6) -- (2,3) -- (3,3) -- (3.5,2.6) -- (4,3) -- (5,3) -- (5.5,2.6) -- (6,3) --  (7,3) -- (7.5,2.6) -- (8,3) -- (9,3) -- (9.5,2) -- (9.5,1) -- (9,0) -- (8,0) -- (7.5,.4) -- (7,0) -- (6,0) -- (5.5,.4) -- (5,0) -- (4,0) -- (3.5,.4) -- (3,0) -- (2,0) -- (1.5,.4) -- (1,0) -- (.5,0);

\draw (1.5,2) circle (.35cm);
\draw (1.5,1) circle (.35cm);
\draw (3.5,2) circle (.35cm);
\draw (3.5,1) circle (.35cm);
\draw (5.5,2) circle (.35cm);
\draw (5.5,1) circle (.35cm);
\draw (7.5,2) circle (.35cm);
\draw (7.5,1) circle (.35cm);

\begin{scope}[very thick, blue]

\draw (1.5,2.72) -- (1.5,2.35);
\draw (1.5,1.65) -- (1.5,1.35);
\draw (1.5,.65) -- (1.5,.28);

\end{scope}

\begin{scope}[very thick, blue, xshift = 2cm]

\draw (1.5,2.72) -- (1.5,2.35);
\draw (1.5,1.65) -- (1.5,1.35);
\draw (1.5,.65) -- (1.5,.28);

\end{scope}

\begin{scope}[very thick, blue, xshift = 4cm]

\draw (1.5,2.72) -- (1.5,2.35);
\draw (1.5,1.65) -- (1.5,1.35);
\draw (1.5,.65) -- (1.5,.28);

\end{scope}

\begin{scope}[very thick, blue, xshift = 6cm]

\draw (1.5,2.72) -- (1.5,2.35);
\draw (1.5,1.65) -- (1.5,1.35);
\draw (1.5,.65) -- (1.5,.28);

\end{scope}

\end{scope}

\end{tikzpicture}$$
\caption{The connected sum of four left-handed trefoils and its all-$A$ state. The all-$A$ state graph of this diagram is four triangles glued at a vertex.}
\label{figure:4trefoils}
\end{figure}

\begin{figure}[h]
\begin{tabular}{|| c || c | c | c | c | c | c | c | c | c | c | c | c | c ||}
\hline
\hline
$j\backslash i$ & -12 & -11 & -10  & -9 & -8 & -7 & -6 & -5 & -4 & -3 & -2 & -1 & 0 \\
 \hline
 \hline
 -7 & & & & & & & & & & & & & 1\\
  \hline
-9 & & & & & & & & & & & & & 1 \\
 \hline
 -11 & & & & & & & & & & & 4& & \\
 \hline
 -13 & & & & & & & & & & & $4_2$& & \\
 \hline
  -15  & & & & & & & & & 6 & 4& & & \\
 \hline
   -17  & & & & & & & &  6 & $6_2$ & & & &  \\
 \hline
    -19  & & & & & & &4 & 6,$6_2$ & & & & & \\
 \hline
    -21   & & & & & & 8 & 6,$4_2$ &  & & & & & \\
 \hline
   -23     & & & & & 5& 4,$8_2$ &  & & & & & & \\
 \hline
      -25   & & & & 3& 8,$5_2$ & & & & & & & & \\
 \hline
       -27   & & &3  & 5,$3_2$ &  & & & & & & & & \\
 \hline
       -29   & & 1 &3,$3_2$ & & & & & & & & &  &\\
 \hline
      -31      &  &3,$1_2$ & & & & & & & & & & & \\
 \hline
         -33    & 1 & & & & & & & & & & & & \\
 \hline
 \hline
\end{tabular}
\caption{The Khovanov homology of the connected sum of four left-handed trefoils.}
\label{figure:KhTref}
\end{figure}

\begin{proof}[Proof of Theorem \ref{theorem:KhTorsion}]
Let $D$ be a link diagram, and let $G_A(D)$ be its all-$A$ state graph. Theorem \ref{theorem:ChromaticKhovanov} implies that the torsion of $Kh^{i,j}(D)$ is isomorphic to the torsion of $H^{p,q}(G_A(D))$ when $-c_- \leq i \leq -c_- + g -1$. Theorem \ref{theorem:ChromaticTorsion} implies that $H(G)$ has only torsion of order two. Thus the Khovanov homology of $D$ also has only torsion of order two in the specified bigradings.
\end{proof}

Theorem \ref{theorem:KhTorsion} deals with the first few homological gradings in Khovanov homology. One can ask similar questions about the first polynomial grading in Khovanov homology. Khovanov \cite{Kh:Patterns} proves that if a link is $A$-adequate, i.e. $G_A(D)$ has no loops for some diagram $D$, then the Khovanov homology in the first polynomial grading is isomorphic to $\mathbb{Z}$. Gonz\'alez-Meneses, Manch\'on, and Silvero \cite{MMS:GeomExtremeKh} associate to each link diagram a simplicial complex whose homology is isomorphic to the Khovanov homology of the link in the first polynomial grading of the Khovanov complex of the diagram. The Khovanov homology in this grading can be trivial, and so we refer to it as a potential extreme polynomial grading. Przytycki and Silvero \cite{PS:ExtremeKh} conjecture that the Khovanov homology in its potential extreme polynomial grading is torsion-free and prove the result for several special cases.

\section{Odd Khovanov homology}
\label{section:oddKh}

In this final section, we prove Theorem \ref{theorem:OddTorsionFree}, an analog of Theorem \ref{theorem:KhTorsion} for odd Khovanov homology. The results of this section are independent from the rest of the paper, but are similar in spirit.

\subsection{Construction of $Kh_{\text{odd}}(L)$}

For each vertex $I\in\mathcal{V}(n)$ in the hypercube, we again let $D(I)$ denote the Kauffman state associated to the vertex $I$. Define $V(D(I))$ to be the free $R$-module generated by the variables $X_1^I,X_2^I,\dots,X_{|D(I)|}^I$. Let $\Lambda^*(V(D(I))) = \Lambda^0(V(D(I))\oplus \Lambda^1(V(D(I))\oplus\cdots\oplus \Lambda^{|D(I)|}(V(D(I))$ be the exterior algebra of $V(D(I))$. For $0\leq r \leq |D(I)|$, the summand $\Lambda^r (V(D(I)))$ is in bigrading $(0,|D(I)|-2r)$. Define $C_{\text{odd}}(D(I))=\Lambda^*(V(D(I)))[h(I)]\{h(I)\}$ and define 
$$CKh_{\text{odd}}(D) = \bigoplus_{I\in\mathcal{V}(n)} C_{\text{odd}}(D(I)) [-c_-]\{c_+-2c_-\}$$
 where $c_\pm$ is the number of positive or negative crossings in $D$ respectively. 

In order to define the differential in odd Khovanov homology, we introduce some additional structure on the Kauffman states of the link diagram $D$. Arbitrarily choose an orientation for each trace in the all-$A$ Kauffman state of $D$. The orientation on the traces of the all-$A$ state induce orientations on traces of all other states as follows. If the trace comes from an $A$-resolution, then it has the same orientation as it did in the all-$A$ resolution, and if the trace comes from a $B$-resolution, then the orientation is induced by rotating the oriented trace at the same crossing with an $A$-resolution $90^\circ$ clockwise. See Figure \ref{figure:resolution}.

Let $\xi$ be an edge from vertex $I$ to vertex $J$. Suppose that $|D(J)| = |D(I)|-1$. Then $C_{\text{odd}}(D(J))\cong C_{\text{odd}}(D(I))/(X_{i_1}^I-X_{i_2}^I)$ where $X_{i_1}^I$ and $X_{i_2}^I$ are the variables associated to the two components of $D(I)$ being merged together. Define $m_{\text{odd}}$ to be the composition $C_{\text{odd}}(D(I)) \to C_{\text{odd}}(D(I))/(X_{i_1}^I-X_{i_2}^I) \xrightarrow{\cong} C_{\text{odd}}(D(J))$, where the first map in the composition is the canonical projection.

Now suppose that $|D(J)|=|D(I)|+1$. Let $X_{i_1}^J$ and $X_{i_2}^J$ be the two generators of $C_{\text{odd}}(D(J))$ corresponding to the components of $D(J)$ that $\xi$ splits such that the trace points from $X_{i_1}^J$ to $X_{i_2}^J$. For each generator $X_k^I$ of $C_{\text{odd}}(D(I))$, define $\Delta_{\text{odd}}(X_k^I) = (X_{i_1}^J - X_{i_2}^J) \wedge X^J_{\eta(k)}$ where $\eta$ is the correspondence between components in $D(I)$ and $D(J)$. If the $k$-th circle is being split, then we can equivalently choose $\eta(k)$ to be $i_1$ or $i_2$. 

If one defines the edge maps $d_\xi$ to simply be $m_{\text{odd}}$ or $\Delta_{\text{odd}}$, then each square in $\{0,1\}^n$ either commutes, anti-commutes, or both. In order to construct a differential, we need each square to anti-commute. Ozsv\'ath, Rasmussen, and Szab\'o \cite{ORS:OddKh} prove that there exists a function $\epsilon\co \mathcal{E}(n)\to\{\pm 1\}$, called a {\em sign assignment}, such that if each edge map is multiplied by $\epsilon$, then every square in the hypercube anti-commutes, thus giving a differential. Moreover, they show that different sign assignments give isomorphic homologies. Define the edge map $d_\xi$ to be $\epsilon(\xi) m_{\text{odd}}$ if $|D(J)| = |D(I)|-1$ or $\epsilon(\xi) \Delta_{\text{odd}}$ if $|D(J)|=|D(I)|+1$, and define 
$d^i_{\text{odd}}\co CKh_{\text{odd}}^{i,*}(D)\to CKh_{\text{odd}}^{i+1,*}$ by $d^i_{\text{odd}} = \sum_{|\xi|=i}d_\xi$. Then $(CKh_{\text{odd}}(D),d_{\text{odd}})$ is a chain complex whose homology is the {\em odd Khovanov homology} $Kh_{\text{odd}}(D;R)$ with coefficients in $R$.

\subsection{Odd Khovanov results}

Corollary \ref{corollary:SameKh} states that if two diagrams $D$ and $D'$ have the same all-$A$ state graphs, then there is an isomorphism of their Khovanov homologies in certain gradings. The proof uses the fact that $Kh(D)$ and $Kh(D')$ are both isomorphic to the chromatic homology of their common all-$A$ state graph. As of yet, there is no odd version of chromatic homology, and so we take a different approach with the following theorem.
\begin{theorem}
\label{theorem:oddKhribbon}
Let $D$ and $D'$ be link diagrams whose all-$A$ state graphs are isomorphic and have girth $g$ of at least two.  Let $c_\pm(D)$ and $c_\pm(D')$ be the number of positive and negative crossings in $D$ and $D'$ respectively. Then $Kh_{\text{odd}}^{i,j}(D)\cong Kh_{\text{odd}}^{p,q}(D')$ for $-c_-(D)\leq i \leq -c_-(D) + g - 1$ and all $j$ where $p-c_-(D')= i- c_-(D)$ and $q +c_+(D') - 2c_-(D') = j+c_+(D) - 2c_-(D).$ 
\end{theorem}
\begin{proof}
Let $G$ be the common all-$A$ state graph of $D$ and $D'$. Suppose that $D$ has $n$ crossings, and let $I\in\mathcal{V}(n)$ be a vertex in the hypercube $\{0,1\}^n$ such that $h(I) < g$. There is a canonical one-to-one correspondence between the components of $D(I)$ and the components of $G(I)$, constructed as follows. The construction of $G$ gives a natural correspondence between the components of the all-$A$ resolution of $D$ and the vertices of $D$. Every edge with height less than $g$ corresponds to merging two components of the Kauffman state or spanning subgraph. Such an edge in the hypercube merges two components of a Kauffman state if and only if adding the corresponding edge to the spanning subgraph of $G$ merges the corresponding components of the subgraph. There is an analogous one-to-one correspondence between the components of $D'(I)$ and the components of $G(I)$. Let $\phi$ be the induced one-to-one correspondence between the components of $D(I)$ and $D'(I)$. Hence one may consider the modules $V(D(I))$ and $V(D'(I))$ both as being free $R$-modules generated by $X_1^I, X_2^I,\dots, X_{|D(I)|}^I$, and therefore the $R$-modules $C_{\text{odd}}(D(I))$ and $C_{\text{odd}}(D'(I))$ are the same. 

Suppose that $G$ has $s$ vertices, and let $\gamma_1,\dots, \gamma_s$ be the components of the all-$A$ state $s_A(D)$ of $D$. Then $\phi(\gamma_1),\dots, \phi(\gamma_s)$ are the components of the all-$A$ state $s_A(D')$ of $D'$. Choose orientations on the traces of $s_A(D)$ and $s_A(D')$ so that if there is a trace between $\gamma_{i_1}$ and $\gamma_{i_2}$ (or $\phi(\gamma_{i_1})$ and $\phi(\gamma_{i_2})$) where $i_1 < i_2$, then the trace points from $\gamma_{i_1}$ to $\gamma_{i_2}$ (or from $\phi(\gamma_{i_1})$ and $\phi(\gamma_{i_2})$). 

If $\xi$ is an edge in $\{0,1\}^n$ with $|\xi| < g$ that merges two components $\gamma_{i_1}$ and $\gamma_{i_2}$ of $D(I)$, then $\xi$ also merges $\phi(\gamma_{i_1})$ and $\phi(\gamma_{i_2})$ of $D'(I)$. Since the traces are all oriented the same in both complexes, the edge maps $d_\xi\co C_{\text{odd}}(D(I))\to C_{\text{odd}}(D(J))$ and $d'_\xi\co  C_{\text{odd}}(D'(I))\to C_{\text{odd}}(D'(J))$ are the same if one ignores sign assignments. However, the sign assignments on the respective cubes may cause $d_\xi = \pm d'_\xi$.

Let $\epsilon$ and $\epsilon'$ be sign assignments for $CKh_{\text{odd}}(D)$ and $CKh_{\text{odd}}(D')$ respectively. Since every edge with $|\xi| < g$ corresponds to a multiplication, it follows that every square containing only such edges has edge maps that commute. Thus both $\epsilon$ and $\epsilon'$ assign an odd number of negative signs to each square in the hypercube with $|\xi| < g$. Let $\{0,1\}^{<g}$ be the hypercube with all edges and vertices of height $g$ or greater removed. 

If a disk is attached to each square of $\{0,1\}^{<g}$, then the resulting space is a disk. Consider the sign assignments as $1$-cochains in $\operatorname{Hom}(C_1,\mathbb{Z}_2)$ where $C_1$ is the space of $1$-chains on $\{0,1\}^{<g}$. Since both $\epsilon$ and $\epsilon'$ assign $-1$ to an odd number of edges around each square, it follows that $\epsilon\cdot\epsilon^\prime$ is a $1$-cocycle. Because the disk is contractible, the product of the edge assignments $\epsilon\cdot\epsilon^\prime$ is the coboundary of a $0$-cochain, that is there exists $\eta\co \mathcal{V}(n)\to\{\pm 1\}$ such that $\eta(I)\eta(J)=\epsilon(\xi)\epsilon^\prime(\xi)$ if $\xi$ is an edge between vertices $I$ and $J$. Therefore, $\epsilon(\xi)=\epsilon'(\xi)$ if and only if $\eta(I)=\eta(J)$. 

Define $\psi\co CKh_{\text{odd}}^{<g,*}(D)\to CKh_{\text{odd}}^{<g,*}(D')$ to be the map on complexes induced by the map $\eta(I)\cdot\operatorname{Id}\co C_{\text{odd}}(D(I))\to C_{\text{odd}}(D'(I))$ for each $I$ with $h(I) < g$. Then $\psi$ is an isomorphism of chain complexes and the result follows.

\end{proof}

\begin{proof}[Proof of Theorem \ref{theorem:OddTorsionFree}]
Since $G$ is planar, there exists an alternating diagram $D'$ whose all-$A$ state graph is $G$. Since $D'$ is alternating, its odd Khovanov homology $Kh_{\text{odd}}(D')$ is torsion-free by \cite{ORS:OddKh}. By Theorem \ref{theorem:oddKhribbon}, the odd Khovanov homology of $D$ is isomorphic to the odd Khovanov homology of $D'$ in the specified bigradings.
\end{proof}

\bibliography{torsionbib}{}

\newcommand{\etalchar}[1]{$^{#1}$}
\providecommand{\bysame}{\leavevmode\hbox to3em{\hrulefill}\thinspace}
\providecommand{\MR}{\relax\ifhmode\unskip\space\fi MR }
\providecommand{\MRhref}[2]{%
  \href{http://www.ams.org/mathscinet-getitem?mr=#1}{#2}
}
\providecommand{\href}[2]{#2}
\begin{thebibliography}{GMMS15}

\bibitem[AP04]{AP:TorsionThickness}
Marta~M. Asaeda and J{\'o}zef~H. Przytycki, \emph{Khovanov homology: torsion
  and thickness}, Advances in topological quantum field theory, NATO Sci. Ser.
  II Math. Phys. Chem., vol. 179, Kluwer Acad. Publ., Dordrecht, 2004,
  pp.~135--166.

\bibitem[BN02]{BN:Kh}
Dror Bar-Natan, \emph{On {K}hovanov's categorification of the {J}ones
  polynomial}, Algebr. Geom. Topol. \textbf{2} (2002), 337--370 (electronic).

\bibitem[BN07]{BN:FastKh}
\bysame, \emph{Fast {K}hovanov homology computations}, J. Knot Theory
  Ramifications \textbf{16} (2007), no.~3, 243--255.

\bibitem[BNG]{BNG:JavaKh}
Dror Bar-Natan and Jeremy Green, \emph{Java{K}h}, a fast program for computing
  {K}hovanov homology, part of the {K}not{T}heory {M}athematica {P}ackage.

\bibitem[CCR08]{CCR:Knight}
Michael Chmutov, Sergei Chmutov, and Yongwu Rong, \emph{Knight move in
  chromatic cohomology}, European J. Combin. \textbf{29} (2008), no.~1,
  311--321.

\bibitem[GMMS15]{MMS:GeomExtremeKh}
J.~Gonz\'alez-Meneses, P.M.G. Manch\'on, and M.~Silvero, \emph{A geometric
  description of the extreme {K}hovanov cohomology}, arXiv:1511.05845, 2015.

\bibitem[HGPR06]{HPR:Torsion}
Laure Helme-Guizon, J{\'o}zef~H. Przytycki, and Yongwu Rong, \emph{Torsion in
  graph homology}, Fund. Math. \textbf{190} (2006), 139--177.

\bibitem[HGR05]{HGR:Chromatic}
Laure Helme-Guizon and Yongwu Rong, \emph{A categorification for the chromatic
  polynomial}, Algebr. Geom. Topol. \textbf{5} (2005), 1365--1388.

\bibitem[HGR12]{HR:Arbitrary}
\bysame, \emph{Khovanov type homologies for graphs}, Kobe J. Math. \textbf{29}
  (2012), no.~1-2, 25--43.

\bibitem[HTW98]{HT:Table}
Jim Hoste, Morwen Thistlethwaite, and Jeff Weeks, \emph{The first 1,701,936
  knots}, Math. Intelligencer \textbf{20} (1998), no.~4, 33--48.

\bibitem[Kho00]{Kh:CategorificationJones}
Mikhail Khovanov, \emph{A categorification of the {J}ones polynomial}, Duke
  Math. J. \textbf{101} (2000), no.~3, 359--426.

\bibitem[Kho03]{Kh:Patterns}
\bysame, \emph{Patterns in knot cohomology. {I}}, Experiment. Math. \textbf{12}
  (2003), no.~3, 365--374.

\bibitem[Lee05]{Lee:Endo}
Eun~Soo Lee, \emph{An endomorphism of the {K}hovanov invariant}, Adv. Math.
  \textbf{197} (2005), no.~2, 554--586.

\bibitem[McC01]{M:Spectral}
John McCleary, \emph{A user's guide to spectral sequences}, second ed.,
  Cambridge Studies in Advanced Mathematics, vol.~58, Cambridge University
  Press, Cambridge, 2001.

\bibitem[MPS{\etalchar{+}}17]{MPSWY:KhTorsion}
Sujoy Mukherjee, J\'ozef Przytycki, Marithania Silvero, Xiao Wang, and
  Seung~Yeop Yang, \emph{Search for torsion in {K}hovanov homology},
  arXiv:1701.04924, 2017.

\bibitem[ORS13]{ORS:OddKh}
Peter~S. Ozsv{\'a}th, Jacob Rasmussen, and Zolt{\'a}n Szab{\'o}, \emph{Odd
  {K}hovanov homology}, Algebr. Geom. Topol. \textbf{13} (2013), no.~3,
  1465--1488.

\bibitem[PPS09]{PPS:First}
Milena~D. Pabiniak, J{\'o}zef~H. Przytycki, and Radmila Sazdanovi{\'c},
  \emph{On the first group of the chromatic cohomology of graphs}, Geom.
  Dedicata \textbf{140} (2009), 19--48.

\bibitem[PS14]{PS:Semiadequate}
J{\'o}zef~H. Przytycki and Radmila Sazdanovi{\'c}, \emph{Torsion in {K}hovanov
  homology of semi-adequate links}, Fund. Math. \textbf{225} (2014), no.~1,
  277--304.

\bibitem[PS16]{PS:ExtremeKh}
Jozef~H. Przytycki and Marithania Silvero, \emph{Homotopy type of circle graphs
  complexes motivated by extreme {K}hovanov homology}, arXiv:1608.03002, 2016.

\bibitem[Shu14]{Shumakovitch:Torsion}
Alexander~N. Shumakovitch, \emph{Torsion of {K}hovanov homology}, Fund. Math.
  \textbf{225} (2014), 343--364.

\bibitem[Shu16]{Shumakovitch:Forthcoming}
\bysame, \emph{Torsion in {K}hovanov homology of homologically thin knots},
  2016, Forthcoming.

\bibitem[Sto14]{Sto:Semi}
Alexander Stoimenow, \emph{On the crossing number of semiadequate links}, Forum
  Math. \textbf{26} (2014), no.~4, 1187--1246.

\bibitem[Tur06]{Turner:KhTwo}
Paul~R. Turner, \emph{Calculating {B}ar-{N}atan's characteristic two {K}hovanov
  homology}, J. Knot Theory Ramifications \textbf{15} (2006), no.~10,
  1335--1356.

\bibitem[Vir04]{Viro:Kh}
Oleg Viro, \emph{Khovanov homology, its definitions and ramifications}, Fund.
  Math. \textbf{184} (2004), 317--342.

\end{thebibliography}
\bibliographystyle {amsalpha}

\end{document}